\documentclass[a4paper,10pt,twoside]{amsart}
\usepackage[english]{babel}
\usepackage[utf8]{inputenc}
\usepackage[a4paper,inner=2.2cm,outer=2.2cm,top=4cm,bottom=4cm,pdftex]{geometry}
\usepackage{fancyhdr}
\usepackage{color}
\usepackage{bold-extra}
\usepackage{mathrsfs}

\usepackage{comment}
\usepackage{graphics}
\usepackage{aliascnt}
\usepackage[pdftex,citecolor=green,linkcolor=red]{hyperref}
\usepackage{mathtools}
\usepackage{mathdots}

\usepackage{amsmath}
\usepackage{amsfonts}
\usepackage{amssymb}
\usepackage{amsthm}
\usepackage{enumerate}
  
\newtheorem{theorem}{Theorem}[section]
\newtheorem{corollary}[theorem]{Corollary}

\newtheorem{lemma}[theorem]{Lemma}
\newtheorem{prop}[theorem]{Proposition}
\theoremstyle{definition}

\newtheorem{rem}[theorem]{Remark}

\numberwithin{equation}{section}

\newcommand\R{\mathbb{R}}
\newcommand\Z{\mathbb{Z}}
\newcommand\N{\mathbb{N}}
\newcommand\C{\mathbb{C}}
\newcommand\Q{\mathbb{Q}}

\newcommand\A{\mathbb{A}}

\newcommand\eps{\varepsilon}
\renewcommand{\Re}{\textnormal{Re}}
\renewcommand{\Im}{\textnormal{Im}}

\usepackage{tikz}
\usetikzlibrary{calc}
\newlength{\lyxlabelwidth}      % auxiliary length 
	{\settowidth{\lyxlabelwidth}{#2}
		\begin{description}[font=\normalfont,style=sameline,
			leftmargin=\lyxlabelwidth,#1]}
	{\end{description}}

\begin{document}

\title{On the Number of Hecke Eigenvalues of the Same Sign on $\mathrm{GL}_n$}

\author{Jesse J\"a\"asaari}
\address{Department of Mathematics and Statistics\\
  P.O. Box 68, 00014\\
  University of Helsinki, Finland}
  \email{jesse.s.jaasaari@helsinki.fi}
  

\maketitle  
  
\begin{abstract}
We study the distribution of signs of the (real-valued) Hecke eigenvalues $A(m,1,...,1)$ of self-dual Hecke--Maass cusp forms for the group $\mathrm{SL}_n(\Z)$, where $n\geq 2$ is an integer. Our main result establishes, under the Generalised Ramanujan--Petersson Conjecture, that for all $X\leq x\leq 2X$ outside a set of measure $o(X)$ the short interval $(x,x+H]$ contains at least $H(\log X)^{1/n^2-1}$ indices with both positive and negative Hecke eigenvalues, provided that $(\log X)^{1-1/n^2+\eps}\leq H\leq X/10$. For $\mathrm{GL}_2$ Hecke--Maass cusp forms we obtain a stronger bound unconditionally, and for self-dual $\mathrm{GL}_3$ forms we obtain the same logarithmic exponent unconditionally. We also establish results under weaker bounds towards the Generalised Ramanujan--Petersson Conjecture together with the Generalised Lindelöf Hypothesis, and obtain lower bounds for the number of sign changes and for the number of Hecke eigenvalues of each sign in long intervals.
\end{abstract}

\section{Introduction}

\noindent The distribution of Hecke eigenvalues is a central question in the theory of automorphic forms. Over the past few decades the study of statistical properties (such as sizes and sign changes) of Hecke eigenvalues of automorphic forms has attracted much attention, see e.g. \cite{Kohnen2007, KLS2008, Wu2009, LLW2010, KLSW2010, Matomaki2012, Matomaki-Radziwill2015}. The interest in the signs of these eigenvalues partly stems from their connections to other questions such as those concerning the distribution of zeroes of automorphic forms \cite{Ghosh-Sarnak2012, Lester-Matomaki-Radziwill2018, Jaasaari2026, Jaasaari2026+} and partly from their role as higher rank analogues of real characters. Let $f$ be a Maass cusp form for the group $\mathrm{SL}_n(\Z)$ with Fourier coefficients $A(m_1,...,m_{n-1})$. Assume further that the form is an eigenfunction of the full Hecke algebra and is normalised so that $A(1,...,1)=1$. In this case the eigenvalue under the $m^{\text{th}}$ Hecke operator is given by $A(m,1,...,1)$.  

In this paper we focus on the signs of Hecke eigenvalues. When $n>2$ the Hecke eigenvalues $A(m,1,...,1)$ are not necessarily real-valued, but this property is satisfied when the underlying form is self-dual. We impose this assumption throughout the paper. For most of the paper, we write $A(m):=A(m,1,...,1)$ for the Hecke eigenvalues\footnote{Often we emphasise the underlying form $f$ by writing $A_f(m)$.}. Rather than studying sign changes, the central aim of this work is to quantitatively understand how many Hecke eigenvalues of each sign a given interval contains.

For this purpose we define the counting function
\[ 
\mathcal N_f^\pm(x):=\sum_{\substack{m\leq x\\
\pm A_f(m)>0}}1.
\]  
When $n=2$, Lau and Wu \cite{Lau-Wu2009} have shown that the optimal lower bound $\mathcal N_f^\pm(x)\gg x$ holds in the case where $f$ is a holomorphic cusp form. For Maass cusp forms (with  any rank $n\geq 2$) a weaker lower bound $\mathcal N_f^\pm(x)\gg x^{1-2\vartheta_n}(\log x)^{2/n-2}$ has been obtained by Liu and Wu \cite{Liu-Wu2015} assuming\footnote{In fact their result was unconditional for $2\leq n\leq 4$ at the time their work was published, but since then Hypothesis H has been established by Jiang \cite{Jiang2026}.} Hypothesis H of Rudnick and Sarnak \cite{Rudnick-Sarnak1996}. Here $\vartheta_n$ is the exponent towards the Generalised Ramanujan--Petersson Conjecture for Hecke--Maass cusp forms on $\mathrm{GL}_n$. Using a different method, Lau, Liu, and Wu \cite{LLW2017} have also obtained the bound $\mathcal N_f^\pm(x)\gg x^{1/n}$ for $\mathrm{GL}_n$ Hecke--Maass cusp forms. The main reason for these weaker results in the setting of (higher rank) Maass cusp forms is that the method of $\mathcal B$-free numbers used in the case of holomorphic cusp forms is difficult to generalise as it uses Serre's estimate \cite{Serre1981} on the size of the set $\{p\leq x\,:\,A_f(p)=0\}$ as an input. An analogue of this result for (higher rank) Maass cusp forms is not known and is a challenging open problem. Indeed, assuming the non-lacunarity property
\[
\sum_{\substack{p\\A_f(p)=0}}\frac1p<\infty
\]
(which follows from a generalisation of Serre's result), the $\mathcal B$-free number argument of Lau and Wu \cite{Lau-Wu2009} would give the expected lower bound $\mathcal N_f^{\pm}(x)\gg x$ for arbitrary rank $n$ Hecke--Maass cusp forms. Nevertheless, Lau, Liu, and Wu \cite{LLW2017, Liu-Wu2015} were able to obtain a non-trivial lower bound by different means. 

It is natural to wonder whether it is also possible to show that shorter intervals contain quantitatively many positive and negative Hecke eigenvalues. For holomorphic cusp forms in the classical situation ($n=2$) Lau and Wu were able to obtain such a result for intervals of length $Cx^{1/2}$, where $C>0$ depends on $f$. Indeed, they showed \cite[Theorem 2]{Lau-Wu2009} that there exists a constant $C$ depending on $f$ so that for sufficiently large $x$ and any $\eps>0$ we have\footnote{Throughout the paper, when we write $\mathcal N_f^\pm(\cdots)-\mathcal N_f^\pm(\cdots)$ we assume that the choices $\{+,-\}$ in upper indices coincide.}
\[ 
\mathcal N_f^\pm\left(x+Cx^{1/2}\right)-\mathcal N_f^\pm(x)\gg_\eps x^{1/4-\eps}.\]
\noindent As far as the author is aware, no prior results of this type are known in the higher rank setting. The main goal of the present work is to provide such a result for general Hecke--Maass cusp forms under the Generalised Ramanujan--Petersson Conjecture. Our methods differ substantially from those of Lau and Wu, whose argument does not seem to extend to the higher rank setting. 

Let us set
\[
\Delta_f(X):=\prod_{\substack{p\leq X\\A_f(p)=0}}\left(1-\frac1p\right).
\] 
Observe that (see the proof of Theorem \ref{Main_theorem}) under the Generalised Ramanujan--Petersson Conjecture we have 
\[
\#\left\{m\sim X:\,A_f(m)\neq 0\right\}\asymp_f X\Delta_f(X).
\]
In this paper we write $x\sim X$ for $X\leq x\leq 2X$ and $\#\mathcal X$ for the cardinality of a set $\mathcal X$.
\begin{rem}
The quantity $\Delta_f(X)$ is also the density parameter $\delta(\mathcal N;X)$ appearing in \cite{Matomaki-Radziwill2020} for the multiplicative set\footnote{More generally, we say that a subset $\mathcal N\subset\N$ is \emph{multiplicative} if for any $m,n\geq 1$ with $(m,n)=1$ we have $m,n\in\mathcal N$ if and only if $mn\in\mathcal N$.} $\mathcal N=\{m\in\N:\,A_f(m)\neq 0\}$.
\end{rem}

\noindent Our first result gives precise information about the number of positive and negative Hecke eigenvalues for Hecke--Maass cusp forms for the group $\mathrm{SL}_n(\Z)$ in short intervals $(x,x+H]$ for almost all $x$ in a wide range of $H$. Throughout this work the phrase ``almost all $x\sim X$" means that the Lebesgue measure of the exceptional set is $o(X)$. 

\begin{theorem}\label{Main_theorem}
Let $f$ be a self-dual Hecke--Maass cusp form for the group $\mathrm{SL}_n(\Z)$. Assume the Generalised Ramanujan--Petersson Conjecture and let $1\leq H\leq X/10$ be so that $H\Delta_f(X)\longrightarrow\infty$ as $X\longrightarrow\infty$. Then for all but $o_f(X)$ of $x\sim X$ we have  
%\[
%\mathcal N_f^\pm(x+H)-\mathcal N_f^\pm(x)=\frac H{2X}\#\left\{m\sim X:\,A(m)\neq 0\right\}%+o_f(H\Delta_f(X)).
%\]
%In particular, 
\[
\mathcal N_f^\pm(x+H)-\mathcal N_f^\pm(x)\asymp_f H\Delta_f(X).
\]
%for almost all $x\sim X$.
\end{theorem}
\noindent It is also easy to show (see the end of Section $5$) that under the Generalised Ramanujan--Petersson Conjecture we have $\Delta_f(X)\gg_f (\log X)^{-1+1/n^2}$ and so we have the estimate
\[
\mathcal N_f^\pm(x+H)-\mathcal N_f^\pm(x)\gg_f \frac H{(\log X)^{1-1/n^2}}
\]
for almost all $x\sim X$ in the range $H/(\log X)^{1-1/n^2}\longrightarrow\infty$ as $X\longrightarrow\infty$. 

We are also able to prove results for $\mathrm{GL}_2$ and $\mathrm{GL}_3$ Hecke--Maass cusp forms without assuming the Generalised Ramanujan--Petersson Conjecture.
\begin{theorem}\label{secondMain_theorem}
Let $\varepsilon>0$ be arbitrarily small but fixed.
\begin{enumerate}
\item Let $f$ be a Hecke--Maass cusp form for the group $\mathrm{SL}_2(\Z)$. Then for $(\log X)^{1/2+\varepsilon}\ll_\varepsilon H\leq X/10$ we have  
\[
\mathcal N_f^\pm(x+H)-\mathcal N_f^\pm(x)\gg_f\frac H{\sqrt{\log X}}
\]
for almost all $x\sim X$.
\item Let $f$ be a self-dual Hecke--Maass cusp form for the group $\mathrm{SL}_3(\Z)$. Then for $(\log X)^{8/9+\varepsilon}\ll_\varepsilon H\leq X/10$ we have  
\[
\mathcal N_f^\pm(x+H)-\mathcal N_f^\pm(x)\gg_f\frac H{(\log X)^{8/9}}
\]
for almost all $x\sim X$.
\end{enumerate}
\end{theorem}
\noindent Note that in the $n=2$ case this gives a lower bound in much more general short intervals compared to \cite{Lau-Wu2009} for almost all $x\sim X$. On the other hand, the result\footnote{Their result was formulated for holomorphic cusp forms, but under the Generalised Ramanujan--Petersson Conjecture their method gives an analogous result for $\mathrm{GL}_2$ Hecke--Maass cusp forms.} of Lau and Wu mentioned above applies to every sufficiently large $x$.

\begin{rem}
By inspecting the proof one observes that if the estimate
\[
\sum_{z<p\leq w}\frac{|A_f(p)|^4}p\leq C_4\sum_{z<p\leq w}\frac 1p+O_f\left(\frac1{\log z}\right)
\]
(and little more) holds uniformly in $2\leq z<w$ for some absolute constant $C_4>0$, then the conclusions of the previous theorems hold in the range $H/(\log X)^{1-1/C_4}\rightarrow\infty$. For $\mathrm{GL}_3$ Hecke--Maass cusp forms self-duality allows us to prove $C_4\leq 9$ without assuming the Generalised Ramanujan--Petersson Conjecture using currently known results towards Langlands functoriality. 
\end{rem}

\noindent As a direct consequence of Theorem \ref{Main_theorem} we deduce the following.  

\begin{corollary}\label{Corollary}
Let $f$ be a self-dual Hecke--Maass cusp form for the group $\mathrm{SL}_n(\Z)$. Assume the Generalised Ramanujan--Petersson Conjecture. Then, provided that $(\log X)^{1-1/n^2+\varepsilon}\ll_\varepsilon H\leq X/10$ for an arbitrarily small fixed $\eps>0$, there are $\gg X/H$ pairwise disjoint subintervals of length $\gg H$ of $[X,2X]$ such that if $x$ belongs to any of these subintervals, we have
\begin{align}\label{WZ-type_result}
\mathcal N_f^\pm(x+H)-\mathcal N_f^\pm(x)\gg_f \frac H{(\log X)^{1-1/n^2}}.
\end{align}
\end{corollary}
\noindent Furthermore, by Theorem \ref{secondMain_theorem} the assumption of the Generalised Ramanujan--Petersson Conjecture can be removed for $n\in\{2,3\}$ and the lower bound in (\ref{WZ-type_result}) can be slightly increased for $n=2$. These deductions are made in Section $6$. In the case $n=2$ this significantly improves an earlier result\footnote{Again, technically their result is for holomorphic cusp forms, but the argument can be adapted to the Maass cusp form setting as well.} of Wu and Zhai \cite[Theorems 3 and 4]{Wu-Zhai2009} achieved by a different method. Furthermore, as far as the author is aware, no prior results of this kind existed in the higher rank setting. 

Using similar ideas we are also able to improve the existing bounds for the number of sign changes of Hecke eigenvalues in long intervals. As usual, by number of sign changes of a possibly vanishing sequence
$a(n)$ we mean the number of sign changes of $a(n)$ on the subset of $n$ at which $a(n)\neq 0$. 

\begin{corollary}\label{Sign_changes}
Let $f$ be a self-dual Hecke--Maass cusp form for the group $\mathrm{SL}_n(\Z)$. Assume the Generalised Ramanujan--Petersson Conjecture. Then the number of sign changes of the sequence of Hecke eigenvalues $\{A_f(m)\}_{m\leq X}$ is $\asymp X\Delta_f(X)$. The same result holds unconditionally for $n\in\{2,3\}$. 
\end{corollary}
For $n=2$ this recovers the result of Matom\"aki and Radziwi{\l\l} \cite[Theorem 1.2]{Matomaki-Radziwill2015} and for $n=3$ this significantly improves an earlier result by the author \cite[Theorem 1]{Jaasaari2023}. For $n\geq 4$ this is sharper than the conditional result obtained in \cite{LNW2026} under a subconvexity estimate for the Godement--Jacquet $L$-function in the $t$-aspect and a strong lower bound on non-vanishing of Hecke eigenvalues. 

As a by-product of these arguments we are able to produce lower bounds for the quantity $\mathcal N_f^\pm(x)$ for large enough $x$. The first part of the following result holds for an admissible bound towards the Ramanujan--Petersson Conjecture whereas the latter part gives an improved bound under the full Ramanujan--Petersson Conjecture. Let $\vartheta\geq 0$ be so that $|A_f(p)|\leq np^\vartheta$ for all primes $p$. Under the Generalised Ramanujan--Petersson Conjecture we have $\vartheta=0$.  

\begin{theorem}\label{thirdMain_theorem}
Let $f$ be a self-dual Hecke--Maass cusp form for the group $\mathrm{SL}_n(\Z)$. 
\begin{enumerate}
\item Then we have
\[
\mathcal N_f^\pm(x)\gg \frac{x^{1-2\vartheta}}{\log x}
\]
for sufficiently large $x$.
\item Assume the Generalised Ramanujan--Petersson Conjecture. Then we have  
\[
\mathcal N_f^\pm(x)\gg \frac x{(\log x)^{1-1/n^2}}
\]
for sufficiently large $x$.
\end{enumerate}
\end{theorem}
\noindent The result of Liu and Wu \cite{Liu-Wu2015} states that $\mathcal N_f^\pm(x)\gg x^{1-2\vartheta}(\log x)^{2/n-2}$. Thus part ($1$) improves the lower bound of Liu and Wu by a logarithmic factor when $n\geq 3$; for $n=2$ the two bounds coincide. Under the full Generalised Ramanujan--Petersson Conjecture, part ($2$) gives a further logarithmic improvement.

We can also improve the lower bound for the quantity $\mathcal N_f^\pm(x)$ without assuming the Generalised Ramanujan--Petersson Conjecture for $\mathrm{GL}_2$ and $\mathrm{GL}_3$ forms. 

\begin{theorem}\label{fourthMain_theorem}
\begin{enumerate}
\item Let $f$ be a Hecke--Maass cusp form for the group $\mathrm{SL}_2(\Z)$. Then we have  
\[
\mathcal N_f^\pm(x)\gg_f\frac x{\sqrt{\log x}}
\]
for sufficiently large $x$.

\item Let $f$ be a self-dual Hecke--Maass cusp form for the group $\mathrm{SL}_3(\Z)$. Then we have  
\[
\mathcal N_f^\pm(x)\gg_f\frac x{(\log x)^{8/9}}
\]
for sufficiently large $x$. 
\end{enumerate}
\end{theorem}
\noindent The first part slightly improves the bound $\mathcal N_f^\pm(x)\gg_f x/\log x$, which is a consequence of \cite[Theorem 1.2]{Matomaki-Radziwill2015}. The second part significantly improves an earlier result of Liu and Wu \cite{Liu-Wu2015} who showed $\mathcal N_f^\pm(x)\gg_f x^{2/7}(\log x)^{-4/3}$ for $\mathrm{GL}_3$ Maass cusp forms. Observe also that for a holomorphic cusp form $f$, Lau, Liu, and Wu \cite{LLW2010} have shown that $\mathcal N_{\text{Sym}^2f}^\pm(x)\gg_f x$. However, the method that achieves this is based on $\mathcal B$-free numbers and thus relies on Deligne's bound for the Hecke eigenvalues that is not presently available for Maass cusp forms.

\begin{rem}\label{Improvement}
\begin{enumerate}
\item If one considers self-dual Hecke--Maass cusp forms of higher level, then in the $\mathrm{GL}_2$ case such a form may correspond to a representation of dihedral type. For these forms one can show a slightly weaker bound 
\[
\mathcal N_f^\pm(x)\gg_f\frac x{(\log x)^{3/4}}.
\]
\item In the $\mathrm{GL}_3$ case one can in turn prove a slightly improved bound
\[
\mathcal N_f^\pm(x)\gg_f\frac x{(\log x)^{1/3}}
\]
if the form $f$ is a Gelbart--Jacquet lift of a $\mathrm{GL}_2$ Maass cusp form that is of octahedral type. 
\end{enumerate}
\end{rem}

\noindent We also apply another method based on the moments of sums of Hecke eigenvalues to obtain results that hold under a bound towards the Generalised Ramanujan--Petersson Conjecture, together with the Generalised Lindel\"of Hypothesis for the standard Godement--Jacquet $L$-function and the Rankin--Selberg $L$-function. Here we write $|\mathcal X|$ for the measure of a set $\mathcal X$.

\begin{theorem}\label{fifthMain_theorem}
Let $f$ be a self-dual Hecke--Maass cusp form for the group $\mathrm{SL}_n(\Z)$ and let $\varepsilon>0$ be arbitrarily small but fixed. Assume the weak Ramanujan--Petersson Conjecture in the form $|A(m)|\ll_\eps m^{\vartheta+\eps}$. Furthermore, assume the Generalised Lindel\"of Hypothesis for $L(s,f)$ and the Rankin--Selberg $L$-function $L(s,f\times \widetilde f)$.

(1) Let $1\leq H\leq X/2$. Then there exists a subset $\mathcal E(X,H)$ of $[X,2X]$ with $|\mathcal E(X,H)|\ll_\eps X^{1+2\vartheta+5\eps}/H$ so that for any $x\in[X,2X]\setminus\mathcal E(X,H)$ we have  
\[
\mathcal N_f^\pm(x+H)-\mathcal N_f^\pm(x)\gg_\eps HX^{-2\vartheta-2\eps}. 
\]
In particular, for $H\geq X^{2\vartheta+6\eps}$ the exceptional set has size $o(X)$. 

(2) For $X^{1/2+\vartheta+\eta}\leq H\leq X/2$, with an arbitrarily small fixed $\eta>0$, we have  
\[
\mathcal N_f^\pm(x+H)-\mathcal N_f^\pm(x)\gg_\eps HX^{-2\vartheta-\eps}
\]
for any $x\sim X$. 
\end{theorem}
\noindent Finally, assuming the full Ramanujan--Petersson Conjecture we are able to sharpen the factor $X^{-\eps}$ to a logarithmic factor in part (2) of the previous theorem. 

\begin{theorem}\label{sixthMain_theorem}
Let $f$ be a self-dual Hecke--Maass cusp form for the group $\mathrm{SL}_n(\Z)$ and fix an arbitrarily small $\eps>0$. Suppose that $x^{1/2+\eps}\leq H\ll x$. Assume the Generalised Ramanujan--Petersson Conjecture and the Generalised Lindel\"of Hypothesis for both $L(s,f)$ and $L(s,f\times \widetilde f)$. Then we have  
\[
\mathcal N_f^\pm(x+H)-\mathcal N_f^\pm(x)\gg_\eps\frac H{(\log x)^{2\log n+\eps}}.
\]
\end{theorem}

\section{The strategy}

\noindent We make a few remarks concerning the proofs. The methods used in this paper differ substantially from those used in the previous works concerning Hecke eigenvalues of the same sign and in particular we utilise recent advances in the theory of multiplicative functions and results concerning automorphic representations. 

For Theorem \ref{Main_theorem} our methods largely build upon the work of Matom\"aki and Radziwi{\l\l} \cite{Matomaki-Radziwill2020} concerning multiplicative functions in short intervals. We consider the multiplicative functions $g(m)=\text{sgn}(A_f(m))$ and $r(m)=|g(m)|$, and apply \cite[Corollary 1.1]{Matomaki-Radziwill2020} to reduce the proof to understanding averages of $g(m)$ and $r(m)$ over dyadic intervals. At this stage we are able to use some arguments inspired by the work of Mangerel \cite{Mangerel2023b}. Much of the work involves verifying conditions needed to apply results from the theory of multiplicative functions. In particular, we are required to establish certain prime-density and non-pretentiousness conditions. As a consequence of this result, we also improve (Corollary \ref{Sign_changes}) earlier bounds for the number of sign changes of Hecke eigenvalues unconditionally for $n\in\{2,3\}$ and under the Generalised Ramanujan--Petersson Conjecture for $n\geq 4$. 

To remove the Generalised Ramanujan--Petersson Conjecture when $n\in\{2,3\}$ (Theorem \ref{secondMain_theorem}) the key observation is that such an assumption can be replaced with the following estimate, for some constant $C_4>0$,
\begin{align}\label{cond-est-1}
\sum_{z<p\leq w}\frac{|A_f(p)|^4}p\leq C_4\sum_{z<p\leq w}\frac1p+O_f\left(\frac1{\log z}\right)
\end{align}
that holds uniformly in $2\leq z< w$. As Mangerel notes \cite[Remark 1.3]{Mangerel2023}, such an estimate is not generally available unconditionally for $n\geq 3$. However, we show that using known results from the theory of automorphic forms it is possible to obtain such a result without the Generalised Ramanujan--Petersson Conjecture for self-dual $f$ in the case $n=3$. This relies particularly on the description of self-dual Maass cusp forms for the group $\mathrm{SL}_3(\Z)$ due to Ramakrishnan \cite{Ramakrishnan2014} and the work of Kim--Shahidi \cite{Kim-Shahidi2002} on cuspidality of symmetric powers. 

Theorems \ref{thirdMain_theorem} and \ref{fourthMain_theorem} are direct corollaries of Theorems \ref{Main_theorem} and \ref{secondMain_theorem}, respectively. However, we will provide alternative proofs to showcase other ways to utilise multiplicativity to study questions related to signs of Hecke eigenvalues. The proof of Theorem \ref{thirdMain_theorem} is based on analysing the number of non-vanishing Hecke eigenvalues and multiplicativity of $A_f(m)$ in the form that if $r$ is such that $A_f(r)<0$, then $A_f(m)$ and $A_f(mr)$ have opposite signs for $(m,r)=1$ (assuming $A_f(m)\neq 0$). For part (1), the non-vanishing result needed follows from the automorphic prime number theorem; for part (2) it follows from Wirsing’s theorem using an argument inspired by the work of Mangerel \cite{Mangerel2023b}. Again, removing the assumption of the Generalised Ramanujan--Petersson Conjecture for $n\in\{2,3\}$ (Theorem \ref{fourthMain_theorem}) is based on the observation that such an input can be replaced with, for some constant $\alpha>0$, 
\begin{align}\label{cond-est-2} 
\sum_{\substack{z<p\leq w\\A_f(p)\neq 0}}\frac1p\geq\alpha\sum_{z<p\leq w}\frac1p-O\left(\frac1{\log z}\right),
\end{align}
which can then be verified for $n\in\{2,3\}$. Such an assumption allows us again to apply results from \cite{Matomaki-Radziwill2020}.

Theorem \ref{fifthMain_theorem} relies on estimates for
\[
\sum_{x<m\leq x+H}A_f(m) \quad \text{and} \quad \sum_{x<m\leq x+H}|A_f(m)|^2.
\]
These are needed both for individual $x$ and for averages over $x\sim X$. Such bounds are achieved by standard contour-integration techniques. We assume the Generalised Lindel\"of Hypothesis for simplicity, but a weaker pointwise bound for $L(1/2+it,f)$ should also suffice. Finally, the results of the previous theorem can be somewhat sharpened (Theorem \ref{sixthMain_theorem}) if one assumes the Generalised Ramanujan--Petersson Conjecture. This is a fairly straightforward consequence of Shiu's bound and the observation that the Generalised Ramanujan--Petersson Conjecture implies the estimate $|A_f(p)|^q\leq n^{q-2}|A_f(p)|^2$ for any $q>2$.

\begin{rem}
Observe that there is a close relationship between estimates (\ref{cond-est-1}) and (\ref{cond-est-2}). Indeed, using an elementary inequality (which holds for any $C_4>0$ and $u\in\R$)
\begin{align}\label{elem-ineq}
1_{\{u\neq 0\}}\geq\frac{2u}{C_4}-\frac{u^2}{C_4^2}
\end{align}
with $u=|A_f(p)|^2$ and (\ref{uniformRS}), we have that (\ref{cond-est-1}) implies (\ref{cond-est-2}) for $\alpha=1/C_4$. 
\end{rem}

\subsection{Acknowledgements}
This work was supported by the Finnish Cultural Foundation and the Emil Aaltonen Foundation. The author wishes to thank Kaisa Matom\"aki for interesting discussions concerning sign changes and multiplicative functions over the years, and in particular for pointing out the reference \cite{Matomaki-Radziwill2020}. He is also grateful to Steve Lester for useful suggestions, and to Jesse Thorner for a useful remark on an earlier version of the manuscript, which led to strengthening Theorem \ref{thirdMain_theorem}.

\subsection{Declaration on the use of AI} ChatGPT $5.6$ Sol and ChatGPT $6$ Astra were used to identify relevant literature (which in particular led to locating the reference \cite{Kim-Shahidi2002}) and to proofread the manuscript. The output from proofreading gave rise to Remark \ref{Improvement} and Theorem \ref{sixthMain_theorem}, which were carefully reviewed, verified, and revised by the author, who takes full responsibility for their content. The remaining text and mathematical arguments in this paper are entirely human-generated.

\subsection{Organisation of the paper}

\noindent This paper is organised as follows. Section $3$ introduces the necessary notation. In Section $4$ we gather the tools needed in the proofs. Theorem \ref{Main_theorem} is proved in Section $5$. Section $6$ is devoted to the proof of Theorem \ref{secondMain_theorem} and the deduction of Corollaries \ref{Corollary} and \ref{Sign_changes}. Section $7$ contains the proofs of Theorems \ref{thirdMain_theorem} and \ref{fourthMain_theorem}. The final section completes the proofs of Theorems \ref{fifthMain_theorem} and \ref{sixthMain_theorem}. 

\section{Notation}

\noindent We use standard asymptotic notation. If $f$ and $g$ are complex-valued functions defined on some set, say $\mathcal G$, then we write $f\ll g$ to signify that $|f(x)|\leqslant C|g(x)|$ for all $x\in\mathcal G$ for some implicit constant $C\in\mathbb R_+$. The notation $O(g)$ denotes a quantity that is $\ll g$, and $f\asymp g$ means that both $f\ll g$ and $g\ll f$. We write $f=o(g)$ if $g$ never vanishes in $\mathcal G$ and $f(x)/g(x)\longrightarrow 0$ as $x\longrightarrow\infty$. Moreover, we write\footnote{This should not be confused with the notation $\ell\sim L$ used also for $L\leq \ell\leq 2L$ in this paper.} $f\sim g$ if $f(x)/g(x)\longrightarrow 1$ as $x\longrightarrow\infty$. The letter $\varepsilon$ denotes a positive real number, whose value can be fixed to be arbitrarily small, and whose value can be different in different instances in a proof unless stated otherwise. All implicit constants are allowed to depend on $\varepsilon$, on the implicit constants appearing in the assumptions of theorem statements, and on anything that has been fixed. When necessary, we will use subscripts $\ll_{\alpha,\beta,...},O_{\alpha,\beta,...}$, etc. to indicate when implicit constants are allowed to depend on quantities $\alpha,\beta,...$

Let us also write $1_\mathcal A(x)$ for the indicator function of the set $\mathcal A$. For a subset $\mathcal B\subset[X,2X]$ we write $\mathcal B^c$ for the complement of $\mathcal B$ relative to $[X,2X]$. Furthermore, $\Re(s)$ and $\Im(s)$ are the real and imaginary parts of $s\in\mathbb C$, respectively, and occasionally we write $\sigma$ for $\Re(s)$. As usual, $\mu$ denotes the M\"obius function, $\A_\Q$ denotes the adele ring of $\Q$, and $\text{sgn}(u)$ denotes the sign of $u\in\R$;
\begin{align*}
\text{sgn}(u):=\begin{cases}
1 & \text{if }u>0 \\
0 & \text{if }u=0 \\
-1 & \text{if }u<0.
\end{cases}
\end{align*}
Also, $\zeta$ is the Riemann zeta function, $d_k(m)$ is the $k$-fold divisor function, $\Gamma$ denotes the Gamma function,  and $\gamma$ is the Euler--Mascheroni constant.      
        
\section{Preliminaries}

\subsection{Automorphic forms}

Let $f$ be a Hecke--Maass cusp form of type $(\nu_1,...,\nu_{n-1})\in\C^{n-1}$ for the group $\mathrm{SL}_n(\Z)$ with Fourier coefficients $A(m_1,...,m_{n-1})$. If the form is normalised so that $A(1,...,1)=1$, then the eigenvalue under the $m^{\text{th}}$ Hecke operator is given by $A(m,1,...,1)$. The key property is that the Hecke eigenvalues $A(m,1,...,1)$ are multiplicative \cite[Theorem 9.3.11]{Goldfeld2006}.

Next we define the notion of a dual Maass cusp form. Let
\begin{align*}
\widetilde f(z):=f(w\cdot {}^t(z^{-1})w),\qquad\text{where}\qquad w:=\begin{pmatrix} & & & (-1)^{\lfloor\frac n2\rfloor} \\
& & 1 & \\
& \iddots & & \\
1 & & &
\end{pmatrix}. 
\end{align*}
Then $\widetilde f$ is a Maass cusp form of type $(\nu_{n-1},...,\nu_1)\in\C^{n-1}$ for the group $\mathrm{SL}_n(\Z)$ and it is called the dual Hecke--Maass cusp form of $f$. It turns out that 
\[
A_f(m_1,...,m_{n-1})=A_{\widetilde f}(m_{n-1},...,m_1)
\]
for every $m_1,...,m_{n-1}\geq 1$. For Hecke eigenforms a simple application of M\"obius inversion shows that $A(m_1,...,m_{n-1})=\overline{A(m_{n-1},...,m_1)}$. By combining the previous two observations we obtain
\[
\overline{A_f(m,1,...,1)}=A_{\widetilde f}(m,1,...,1).
\]
In particular, if the form $f$ is self-dual ($f=\widetilde f$), then the Hecke eigenvalues are real-valued. 

Associated to such a form $f$ is the $L$-series given by
\[
L(s,f):=\sum_{m=1}^\infty\frac{A(m,1,...,1)}{m^s},
\]
which converges for $\Re(s)>1$. This has an entire continuation to the whole complex plane via the functional equation 
\[
L(s,f)=\pi^{ns-n/2}\frac{G(1-s,\widetilde f)}{G(s,f)}L\left(1-s,\widetilde f\right),
\]
where 
\begin{align*}
G(s,f):=\prod_{j=1}^n\Gamma\left(\frac{s-\lambda_j(\nu)}2\right)\quad\text{and so}\quad G\left(s,\widetilde f\right):=\prod_{j=1}^n\Gamma\left(\frac{s-\widetilde\lambda_j(\nu)}2\right).
\end{align*}
Here $\lambda_j(\nu)$ and $\widetilde\lambda_j(\nu)$ are the Langlands parameters of $f$ and $\widetilde f$, respectively. 

The Generalised Lindel\"of Hypothesis in the $t$-aspect states that on the critical line $\sigma=1/2$ one has the estimate $L(1/2+it,f)\ll_\eps (1+|t|)^\eps$ for any $\eps>0$. The Rankin--Selberg $L$-function of two Hecke--Maass cusp forms $f$ and $g$ for the group $\mathrm{SL}_n(\Z)$ is given by
\[
L(s,f\times g):=\zeta(ns)\sum_{m_1,...,m_{n-1}\geq 1}\frac{A_f(m_1,...,m_{n-1})A_g(m_1,...,m_{n-1})}{(m_1^{n-1}m_2^{n-2}\cdots m_{n-1})^s},
\]
which converges for $\Re(s)>1$. This $L$-series has an analytic continuation to the whole complex plane if $g\neq\widetilde f$ and has a meromorphic continuation to $\C$ with a simple pole at $s=1$ if $g=\widetilde f$. If we set  
\[
\Lambda(s,f\times g):=\prod_{i=1}^n\prod_{j=1}^n\pi^{(-s+\lambda_i(\nu_f)+\lambda_j(\nu_g))/2}\Gamma\left(\frac{s-\lambda_i(\nu_f)-\lambda_j(\nu_g)}2\right)L(s,f\times g),
\]
then the functional equation
\[
\Lambda(s,f\times g)=\Lambda\left(1-s,\widetilde f\times\widetilde g\right)
\]
holds; see \cite[Theorem 12.1.4]{Goldfeld2006}.

The standard $L$-function $L(s,f)$ has an Euler product representation 
\[
L(s,f)=\sum_{m=1}^\infty\frac{A(m,1,...,1)}{m^s}=\prod_p\prod_{j=1}^n\left(1-\alpha_{j,p}(f)p^{-s}\right)^{-1}
\]
for large enough $\Re(s)$, and a similar representation holds for $g$ with parameters $\alpha_{j,p}(g)$. In this case the Rankin--Selberg $L$-function can be written as an Euler product
\[
L(s,f\times g)=\prod_p\prod_{k=1}^n\prod_{\ell=1}^n\left(1-\alpha_{k,p}(f)\alpha_{\ell,p}(g)p^{-s}\right)^{-1}.
\]
Recall that here the complex numbers $\alpha_{j,p}(f)$ are called the Satake parameters of the underlying Hecke--Maass cusp form $f$. Analytic properties of $L(s,f\times\widetilde f)$ imply that
\[
\sum_{m_1^{n-1}m_2^{n-2}\cdots m_{n-1}\leq x}|A(m_1,m_2,...,m_{n-1})|^2\sim \frac{r_f}{\zeta(n)}\cdot x
\]
for some constant $r_f>0$ depending on $f$ \cite[Proposition 12.1.6, Remark 12.1.8]{Goldfeld2006}. As noted in \cite[Theorem 6]{Jaasaari2020}, we also have the estimate
\begin{align}\label{RS-asymp}
\sum_{m\leq x}|A(m,1,...,1)|^2\sim r_f\cdot H_f(1)\cdot x,
\end{align} 
where $H_f(1)>0$.

It was also noted in the proof of \cite[Theorem 6]{Jaasaari2020} that for
\[
D_f(s):=\sum_{m=1}^\infty\frac{|A(m,1,...,1)|^2}{m^s}
\]
we have the factorisation 
\[
D_f(s)=L\left(s,f\times\widetilde f\right)H_f(s)
\]
where $H_f(s)$ is holomorphic and bounded in the half-plane\footnote{An elementary argument shows that $H_f(s)$ is holomorphic and bounded for $\Re(s)>1/2+2\vartheta$, but this can be extended to $\Re(s)>1/2+\vartheta$ using the work of Jiang \cite{Jiang2026}.} $\Re(s)>1/2+\vartheta$. Here $\vartheta$ is an exponent towards the Generalised Ramanujan--Petersson Conjecture explained in the next paragraph.

The asymptotics (\ref{RS-asymp}) can be interpreted as saying that the Hecke eigenvalues $A(m,1,...,1)$ are essentially of constant order of magnitude on average. However, known pointwise bounds for the Hecke eigenvalues are quite far from the expected truth. The Generalised Ramanujan--Petersson Conjecture predicts that an estimate of the form $A(m,1,...,1)\ll_\eps m^\eps$ holds for every $\eps>0$. There are, however, partial results towards this conjecture. Let $\vartheta=\vartheta(n)\geq 0$ be the smallest non-negative real number so that the estimate $A(m,1,...,1)\ll_\eps m^{\vartheta+\eps}$ holds. It is easy to see that $\vartheta\leq 1/2$ \cite[Proposition 12.1.6]{Goldfeld2006}, but currently it is known that $\vartheta\leq 1/2-1/(n^2+1)$. This result is due to Luo, Rudnick, and Sarnak \cite{LRS1999}. For small values of $n$ sharper results are known. We have that $\vartheta(2)\leq 7/64$, $\vartheta(3)\leq 5/14$, and $\vartheta(4)\leq 9/22$ \cite{Kim-Sarnak2003}. The Generalised Ramanujan--Petersson Conjecture predicts that the value $\vartheta(n)=0$ is admissible for every $n\geq 2$. An equivalent estimate holds for the Satake parameters of the underlying form $f$. Namely, we have $|\alpha_{j,p}(f)|\leq p^{\vartheta(n)}$ for every prime $p$. 

Assuming Hypothesis H (which is now known by work of Jiang \cite{Jiang2026}), we have the automorphic prime number theorem, which says that there exists a positive constant $c$ so that 
\[
\sum_{p\leq x}|A_f(p)|^2\log p=x+O\left(xe^{-c\sqrt{\log x}}\right).
\]
By partial summation this gives the uniform estimate
\begin{align}\label{uniformRS}
\sum_{z<p\leq w}\frac{|A_f(p)|^2}p=\sum_{z<p\leq w}\frac1p+O\left(\frac1{\log z}\right)
\end{align}
for $2\leq z<w$. A detailed proof can be found in \cite[Lemma 5.6]{Mangerel2023}. We also have the estimate
\begin{align}\label{uniformbound}
\sum_{z<p\leq w}\frac{A_f(p)}p=O_f\left(\frac1{\log z}\right).
\end{align}
For certain lower rank forms the automorphic prime number theorem holds unconditionally without relying on Hypothesis H. We record the following result of Wu and Ye \cite[Theorem 3]{Wu-Ye2007}.

\begin{lemma}\label{automorphicPNT}
Let $m,m'$ be positive integers. Let $\pi$ be an irreducible unitary cuspidal automorphic representation of $\mathrm{GL}_m(\A_\Q)$ and $\pi'$ be an irreducible unitary cuspidal automorphic representation of $\mathrm{GL}_{m'}(\A_\Q)$ with at least one of them being self-dual.
\begin{enumerate}
\item Suppose that $\max\{m,m'\}\leq 4$. Then we have that 
\begin{align}\label{PNT_rel}
\sum_{p\leq x}A_\pi(p)\overline{A_{\pi'}(p)}\log p=\begin{cases}
\frac{x^{1+i\tau_0}}{1+i\tau_0}+O\left(xe^{-c\sqrt{\log x}}\right) & \text{if }\pi'\simeq \pi\otimes|\det|^{i\tau_0}\quad\text{for some }\tau_0\in\R \\
O\left(xe^{-c\sqrt{\log x}}\right) & \text{otherwise} 
\end{cases}
\end{align}
for some constant $c>0$.
\item Suppose that $\max\{m,m'\}\geq 5$. Assume Hypothesis H for both $\pi$ and $\pi'$; that is 
\[
\sum_p\frac{(\log p)^2|B_\pi(p^\nu)|^2}{p^\nu}<\infty
\]
for any fixed $\nu\geq 2$. Then (\ref{PNT_rel}) holds with error term $O(x/\log x)$. Here
\[
B_\pi(p^\nu):=\sum_{j=1}^m\alpha_{j,p}^\nu.
\]
\end{enumerate}
\end{lemma}
\noindent Let $\pi$ be an irreducible unitary cuspidal automorphic representation of $\mathrm{GL}_2(\A_\Q)$. We call $\pi$ \emph{dihedral} if
\[
\pi\simeq\text{AI}_{K/\mathbb Q}(\chi)
\]
for a Hecke character $\chi$ of a quadratic extension $K/\Q$, where $\chi\neq\chi^\tau$ and $\tau$ is the non-trivial element of $\text{Gal}(K/\Q)$. Here $\text{AI}$ stands for automorphic induction. For non-dihedral $\pi$, we call $\pi$ \emph{tetrahedral} if $\text{Sym}^3\pi$ is not cuspidal, and \emph{octahedral} if $\text{Sym}^3\pi$ is cuspidal but $\text{Sym}^4\pi$ is not cuspidal. These characterisations agree with the standard terminology by \cite[Proposition 3.3.8]{Kim-Shahidi2002}. In particular, $\text{Sym}^4\pi$ is cuspidal whenever $\pi$ is neither dihedral, tetrahedral, nor octahedral \cite[Theorem 3.3.7]{Kim-Shahidi2002}.

When $\pi$ is associated with an irreducible two-dimensional complex Artin representation
\[
\rho:\text{Gal}(\overline{\Q}/\Q)
\longrightarrow\mathrm{GL}_2(\C),
\]
these three types correspond to the image of the associated projective representation $\overline\rho$ being a dihedral group, the alternating group $A_4$, or the symmetric group $S_4$, respectively. In this setting, the representation is of \emph{icosahedral} type if its projective image is isomorphic to the alternating group $A_5$.

%\noindent It is well-known that under the Langlands correspondence irreducible cuspidal automorphic representations of $\mathrm{GL}_2(\A_\Q)$ correspond to two-dimensional continuous Galois representations $\rho:\mathrm{Gal}(\overline{\Q}/\Q)\longrightarrow\mathrm{GL}_2(\C)$. One can also consider the associated projective representations $\overline\rho:\mathrm{Gal}(\overline{\Q}/\Q)\longrightarrow\mathrm{PGL}_2(\C)$. Those irreducible cuspidal automorphic representations of $\mathrm{GL}_2(\A_\Q)$ for which $\overline\rho$ has a finite image fall into distinct types corresponding to finite subgroups of $\mathrm{PGL}_2(\C)$. These types are called \emph{dihedral} (projective image is a dihedral group $D_{2n}$), \emph{tetrahedral} (projective image is the alternating group $A_4$), \emph{octahedral} (projective image is the symmetric group $S_4)$, and \emph{icosahedral} (projective image is the alternating group $A_5$).  

\subsection{Moment bounds}

\noindent We begin by deriving certain second moment bounds under the Generalised Lindel\"of Hypothesis. For a closely related result (whose strategy we follow in the proof) for the $k$-fold divisor function, see \cite[Theorem 1.3]{Baluyot-Castillo2024}.

\begin{lemma}\label{RS_average} 
Let $f$ be a Hecke--Maass cusp form for the group $\mathrm{SL}_n(\Z)$ and suppose that $1\leq H\leq X/2$.
\begin{enumerate}
\item Assume the Generalised Lindel\"of Hypothesis for $L(s,f)$ and $L(s,f\times \widetilde f)$ in the $t$-aspect. Let $r_f$ be as in (\ref{RS-asymp}). Then we have
\[
\int\limits_X^{2X}\left|\sum_{x< m\leq x+H}|A_f(m)|^2-r_fH_f(1)\cdot H\right|^2\,\mathrm d x\ll_\eps HX^{1+2\vartheta+\eps}.
\]
\item Assume the Generalised Lindel\"of Hypothesis for $L(s,f)$ in the $t$-aspect. Then we have  
\[
\int\limits_X^{2X}\left|\sum_{x< m\leq x+H}A_f(m)\right|^2\,\mathrm d x\ll_\eps HX^{1+\eps}.
\]
\end{enumerate}
\end{lemma}

\begin{proof}
Both estimates follow from the same argument. More generally, let $c(m)$ be complex numbers so that 
\[
L(s):=\sum_{m=1}^\infty\frac{c(m)}{m^s}
\]
has a meromorphic continuation to the half-plane $\Re(s)>\sigma_0$ (where $0<\sigma_0<1)$ with a possible simple pole of residue $c_f$ at $s=1$. Furthermore, suppose that $L(s)$ satisfies the estimate $L(\sigma_0+\delta+it)\ll_\eps (1+|t|)^\eps$ in the $t$-aspect for any $\delta>0$. Using an inequality of Saffari and Vaughan \cite{Saffari-Vaughan1977, Goldston-Suriajaya2023} we have 
\begin{align}\label{SV_bound}
\int\limits_X^{2X}\left|\sum_{x< m\leq x+H}c(m)-c_fH\right|^2\,\mathrm d x\ll\frac XH\int\limits_0^{4H/X}\int\limits_{X/2}^{2X}\left|\sum_{x< m\leq (1+\beta)x}c(m)-c_f\beta x\right|^2\,\mathrm d x\,\mathrm d \beta.
\end{align}
Now using Perron's formula and shifting the contour to the line $\Re(s)=\sigma_0+\eps$ we have
\[
\sum_{x< m\leq (1+\beta)x}c(m)-c_f\beta x=\frac1{2\pi i}\int\limits_{(\sigma_0+\eps)}L(s)x^s\frac{(1+\beta)^s-1}s\,\mathrm d s.
\]
Next we observe that 
\[
\frac{(1+\beta)^{\sigma+it}-1}{\sigma+it}\ll\min\left\{\beta,\frac1{1+|t|}\right\}.
\]
Thus smoothly decomposing the interval $[X,2X]$ and using Plancherel's theorem we have 
\begin{align*}
\int\limits_X^{2X}\left|\sum_{x< m\leq (1+\beta)x}c(m)-c_f\beta x\right|^2\,\mathrm d x\ll_\eps X^{1+2\sigma_0+\eps}\int\limits_{-\infty}^\infty|L(\sigma+it)|^2\min\left\{\beta^2,\frac1{1+|t|^2}\right\}\,\mathrm d t.
\end{align*}
Using the pointwise estimate for $L(s)$ we may bound the right-hand side by $O_\eps(X^{1+2\sigma_0+\eps}\beta^{1-\eps})$. Plugging this into (\ref{SV_bound}) yields 
\[
\int\limits_X^{2X}\left|\sum_{x< m\leq x+H}c(m)-c_f\cdot H\right|^2\,\mathrm d x\ll_\eps \frac XH\cdot X^{1+2\sigma_0+\eps}\int\limits_0^{4H/X}\beta^{1-\eps}\,\mathrm d \beta\ll_\eps HX^{2\sigma_0+\eps}.
\]
For part (1) we have $c(m)=|A_f(m)|^2$, in which case $\sigma_0=1/2+\vartheta$ and $c_f=r_fH_f(1)$. For part (2) we have $c(m)=A_f(m)$, $\sigma_0=1/2$, and $c_f=0$. This completes the proof. 
\end{proof}

\noindent The following bounds are also standard. 

\begin{lemma}\label{GLH_conseq}
Let $f$ be a Hecke--Maass cusp form for the group $\mathrm{SL}_n(\Z)$. Assume the Generalised Lindel\"of Hypothesis for $L(s,f)$ and $L(s,f\times \widetilde f)$ in the $t$-aspect. Let $r_f$ be as in (\ref{RS-asymp}). Then we have 
\begin{enumerate}
\item 
\[
\sum_{m\leq x}A_f(m)\ll_\eps x^{1/2+\eps};
\]
\item 
\[
\sum_{m\leq x}|A_f(m)|^2=r_fH_f(1)x+O_{f,\eps}\left(x^{1/2+\vartheta+\eps}\right).
\]
\end{enumerate}
\end{lemma}

\begin{proof}
\noindent Again both estimates follow from the same argument. Let $c(m)$ and $L(s)$ be as in the proof of Lemma \ref{RS_average}. Suppose further that $|c(m)|\ll_\eps m^{\theta+\eps}$ for some $0\leq \theta<1$. To estimate the coefficient sum
\[
\sum_{m\leq x}c(m)
\]
we study its first-order Riesz mean
\[
R(x):=\sum_{m\leq x}c(m)(x-m). 
\]
By Mellin inversion we have 
\[
R(x)=\frac1{2\pi i}\int\limits_{(2)}L(s)\frac{x^{s+1}}{s(s+1)}\,\mathrm d s.
\]
We shift the contour to the line $\Re(s)=\sigma_0+\eps$ with a possible simple pole at $s=1$ contributing the residue $c_f\cdot x^2/2$. Thus
\[
R(x)=c_f\cdot\frac{x^2}2+\frac1{2\pi i}\int\limits_{(\sigma_0+\eps)}L(s)\frac{x^{s+1}}{s(s+1)}\,\mathrm d s.
\]
Now observe that we have
\[
\sum_{m\leq x}c(m)=R(x+1)-R(x)+O_\eps\left(x^{\theta+\eps}\right)
\]
and consequently
\[
\sum_{m\leq x}c(m)=c_f\left(x+\frac12\right)+\frac1{2\pi i}\int\limits_{(\sigma_0+\eps)}L(s)K_x(s)\,\mathrm d s+O_\eps\left(x^{\theta+\eps}\right),
\]
where
\[
K_x(s):=\frac{(x+1)^{s+1}-x^{s+1}}{s(s+1)}.
\]
We have the simple estimate
\[
|K_x(s)|\ll\min\left\{\frac{x^\sigma}{1+|t|},\frac{x^{\sigma+1}}{1+|t|^2}\right\}.
\]
Using these bounds to treat the ranges $|t|\leq x$ and $|t|>x$ separately, together with the bound $L(\sigma_0+\delta+it)\ll_\eps (1+|t|)^\eps$ for any $\delta>0$, in the integral over $\Re(s)=\sigma_0+\eps$ it follows that 
\[
\sum_{m\leq x}c(m)=c_f\left(x+\frac12\right)+O_\eps\left(x^{\max\{\sigma_0,\theta\}+\eps}\right).
\]
To prove (1), we have $c(m)=A_f(m)$, $c_f=0$, $\theta=\vartheta<1/2$, and $\sigma_0=1/2$. Thus we immediately obtain
\[
\sum_{m\leq x}A_f(m)\ll_\eps x^{1/2+\eps},
\]
as claimed.

For (2), we have $c(m)=|A_f(m)|^2$. By \cite{Jaasaari2020} we have $L(s)=L(s,f\times \widetilde f)H_f(s)$, where $H_f(s)$ is bounded and holomorphic in $\Re(s)>1/2+\vartheta$. Furthermore, $D_f(s)$ has a simple pole at $s=1$ (from the pole of the Rankin--Selberg $L$-function at $s=1$) with residue
\[
c_f=H_f(1)\text{Res}_{s=1}L(s,f\times \widetilde f)=r_fH_f(1). 
\]
We also have $\sigma_0=1/2+\vartheta$ and $\theta=2\vartheta<1/2+\vartheta$. Thus, we have shown that 
\[
\sum_{m\leq x}|A_f(m)|^2=r_fH_f(1)x+O_{f,\eps}\left(x^{1/2+\vartheta+\eps}\right),
\]
as desired. This completes the proof. 
\end{proof}

\subsection{Multiplicative functions}

\noindent We now gather results from multiplicative number theory needed in the proofs. The following is \cite[Corollary 1.1]{Matomaki-Radziwill2020}.

\begin{prop}\label{MR}
Let $\mathcal N$ be a multiplicative subset of $\N$. Let $\phi:\N\longrightarrow[-1,1]$ be a multiplicative function. Suppose that there exists a constant $\alpha>0$ such that for all $2\leq z\leq w$,
\begin{align}\label{MR-condition}
\sum_{\substack{z<p\leq w\\p\in\mathcal N}}\frac1p\geq\alpha\sum_{z<p\leq w}\frac1p-O\left(\frac1{\log z}\right).
\end{align}
Then there exists a constant $\kappa=\kappa(\alpha)>0$ such that, for all $\delta\in(0,1/1000)$ and $2\leq h_0\leq X$,
\[
\left|\frac1{h_0}\sum_{\substack{x<m\leq x+h_0\delta(\mathcal N;X)^{-1}\\m\in\mathcal N}}\phi(m)-\frac1{X\delta(\mathcal N;X)}\sum_{\substack{m\sim X\\m\in\mathcal N}}\phi(m)\right|<\delta
\]
outside a subset of $[X,2X]$ of measure $\ll Xh_0^{-\delta^\kappa}$. Here
\[
\delta(\mathcal N;X):=\prod_{\substack{p\leq X\\p\not\in\mathcal N}}\left(1-\frac1p\right). 
\]
\end{prop}
\noindent We also require a version of Wirsing's theorem for non-negative multiplicative functions \cite{Wirsing1967}.

\begin{prop}\label{Wirsing2}
Let $\phi:\N\longrightarrow[0,\infty)$ be a multiplicative function so that
\begin{enumerate}
\item $\text{sup}_p|\phi(p)|\leq A$ for some $A>0$;
\item there exists $\tau>0$ so that 
\[
\sum_{p\leq x}\frac{\phi(p)\log p}p=(\tau+o(1))\log x;
\]
\item \[
\sum_p\sum_{v\geq 2}\frac{\phi(p^v)\log(p^v)}{p^v}<\infty.
\]
\end{enumerate}
Then we have  
\[
\sum_{m\leq x}\phi(m)=\left(\frac{e^{-\gamma\tau}}{\Gamma(\tau)}+o(1)\right)\frac x{\log x}\prod_{p\leq x}\sum_{j=0}^\infty\frac{\phi(p^j)}{p^j}.
\]
\end{prop}

\noindent We also require a result complementing Wirsing's theorem that allows us to say that if
$\phi$ is a multiplicative function such that $|\phi(m)|$ satisfies the hypotheses of the previous proposition, then, provided $\phi$ oscillates sufficiently, we have $\sum_{m\leq x}\phi(m)=o(\sum_{m\leq x}|\phi(m)|)$. The following result of this kind is due to Tenenbaum \cite[Corollary 2.1]{Tenenbaum2017}.

\begin{prop}\label{Tenenbaum}
Let $T\geq 1$ and $\phi:\N\longrightarrow\C$ be a multiplicative function such that
\begin{enumerate}
\item $\text{sup}_p|\phi(p)|\leq A$ for some $A>0$;
\item \[
\sum_p\sum_{\nu\geq 2}\frac{|\phi(p^\nu)|(\log p^\nu)^2}{p^\nu}\leq B
\]
for some $B>0$;
\item there exists a constant $\beta>0$ such that for any $2\leq y\leq x$ we have
\[
\sum_{y\leq p\leq x}\frac{|\phi(p)|}p\geq\beta\log\left(\frac{\log x}{\log y}\right)+O(1).
\] 
\end{enumerate} 
Then, as $x\longrightarrow\infty$,
\[
\left|\frac1x\sum_{m\leq x}\phi(m)\right|\ll_{A,B,\beta}\left(\frac1x\sum_{m\leq x}|\phi(m)|\right)\left(\frac{1+m_\phi(x;T)}{e^{m_\phi(x;T)}}+\frac1{\sqrt T}+\frac1{\log x}\right),
\]
where we have denoted
\[
m_\phi(x;T):=\min_{|t|\leq T}\sum_{p\leq x}\frac{|\phi(p)|-\Re(\phi(p)p^{-it})}p.
\]
\end{prop}
\noindent To estimate the quantity $m_\phi(x;T)$ we will require the following special case of \cite[Lemma 5.1]{Matomaki-Radziwill2020}.

\begin{lemma}\label{pretentious}
Let $\phi:\N\longrightarrow[-1,1]$ be a real-valued multiplicative function that satisfies the estimate
\[
\sum_{z<p\leq w}\frac{|\phi(p)|}p\geq\alpha\sum_{z<p\leq w}\frac1p-O\left(\frac1{\log z}\right)
\]
uniformly in $2\leq z<w$ for some $\alpha>0$. Let $\rho_\alpha:=(\alpha/3)-(2/3\pi)\sin(\pi\alpha/2)$. Then for any $|t|\leq x$ and any $0<\rho<\rho_\alpha$ we have
\[
\sum_{p\leq x}\frac{|\phi(p)|-\Re(\phi(p)p^{-it})}p\geq\rho\min\left\{\log\log x,3\log(1+|t|\log x)\right\}+O_{\alpha,\rho}(1).
\]

\end{lemma}

\noindent The final tool from the theory of multiplicative functions is Shiu's bound \cite{Shiu1980}. 

\begin{lemma}\label{Shiu}
Let $\phi:\N\longrightarrow\C$ be a multiplicative function satisfying $|\phi(m)|\leq d_B(m)^C$ for all $m\leq 2X$, where $B$ and $C$ are positive integers. Let $\sqrt X\leq Y\leq X$, $\delta\in(0,1)$, and let $Y^\delta\leq y\leq Y$. Then
\[
\sum_{Y\leq m\leq Y+y}|\phi(m)|\ll_{B,C,\delta}y\mathcal P_\phi(X),
\]
where 
\[
\mathcal P_\phi(X):=\prod_{p\leq X}\left(1+\frac{|\phi(p)|-1}p\right).
\]
\end{lemma}

\section{Proof of Theorem \ref{Main_theorem}}
\noindent Put $g(m)=\text{sgn}(A_f(m))$ and $r(m)=|g(m)|=1_{A_f(m)\neq 0}$. Clearly both $g$ and $r$ are multiplicative functions, and both vanish if $A_f(m)=0$. Let
\[
\mathcal R_f(x;H):=\sum_{x< m\leq x+H}r(m)=\sum_{\substack{x< m\leq x+H\\A_f(m)\neq 0}}1\qquad\text{and}\qquad\mathcal G_f(x;H):=\sum_{x< m\leq x+H}g(m)=\sum_{x< m\leq x+H}\text{sgn}(A_f(m)).
\]
The goal is to relate these short averages to averages over dyadic intervals $[X,2X]$. Set $h_0=H\Delta_f(X)$ and let $0<\eta<1/1000$ be a parameter chosen later. Applying Proposition \ref{MR} with the choices $\phi=1$ and $\phi=g$ (with $\mathcal N=\{m\in\N:\,A_f(m)\neq 0\}$) gives
\begin{align}\label{Rshort}
\mathcal R_f(x;H)=\frac HX\sum_{m\sim X}r(m)+O(\eta H\Delta_f(X))
\end{align}
and
\begin{align}\label{Gshort}
\mathcal G_f(x;H)=\frac HX\sum_{m\sim X}g(m)+O(\eta H\Delta_f(X))
\end{align}
simultaneously outside a subset of $x\in[X,2X]$ with size $\ll Xh_0^{-\eta^\kappa}$ for some $\kappa=\kappa(n)>0$. To see that the condition (\ref{MR-condition}) holds in this context, note that under the Generalised Ramanujan--Petersson Conjecture we have a simple bound 
\[
1_{\{A_f(p)\neq 0\}}\geq\frac{|A_f(p)|^2}{n^2}.
\]
Together with (\ref{uniformRS}) this gives the required uniform estimate
\begin{align}\label{MR_cond}
\sum_{\substack{z<p\leq w\\ A_f(p)\neq 0}}\frac1p\geq\frac1{n^2}\sum_{z<p\leq w}\frac1p-O_f\left(\frac1{\log z}\right).
\end{align}
Choose $\eta=(\log h_0)^{-1/2\kappa}$. Note that $0<\eta<1/1000$ for sufficiently large $X$ as $h_0\longrightarrow\infty$ along with $X$. With this choice we have that $h_0^{-\eta^\kappa}=\exp(-\sqrt{\log h_0})$. Thus the size of the exceptional set is $\ll_f X\exp(-\sqrt{\log h_0})=o_f(X)$. Now to complete the proof it suffices to show that, under the Generalised Ramanujan--Petersson Conjecture,
\begin{enumerate}
\item \[
\sum_{m\sim X}r(m)\asymp_f X\Delta_f(X);
\]
\item \[
\sum_{m\sim X}g(m)=o_f(X\Delta_f(X)).
\]
\end{enumerate}
Indeed, under these estimates it follows from (\ref{Rshort}) and (\ref{Gshort}) that outside of a subset of $[X,2X]$ of size $o_f(X)$ we simultaneously have 
\[
\mathcal R_f(x;H)\asymp_f H\Delta_f(X) \qquad \text{and} \qquad \mathcal G_f(x;H)=o_f(H\Delta_f(X)). 
\]
This finishes the proof as  
\[
\mathcal N_f^\pm(x+H)-\mathcal N_f^\pm(x)=\frac12\left(\mathcal R_f(x;H)\pm\mathcal G_f(x;H)\right).
\]
%To see that (1) holds, by Proposition \ref{Wirsing2} and a straightforward computation we have  
%\[
%\sum_{m\sim X}r(m)\asymp_f X\prod_{p\leq X}\left(1+\frac{r(p)-1}p\right)
%\]
%as $r(p)\in\{0,1\}$.

%But 
%\begin{align*}
%r(p)-1=\begin{cases}
%0 & \text{if }A_f(p)\neq 0 \\
%-1 & \text{if }A_f(p)=0
%\end{cases}
%\end{align*}
%and so
%\[
%\prod_{p\leq X}\left(1+\frac{r(p)-1}p\right)=\Delta_f(X). 
%\]
\noindent The assertion (1) follows immediately from \cite[(13)]{Matomaki-Radziwill2020} and (\ref{MR_cond}). The argument to establish (2) is slightly more elaborate. For an application of Proposition \ref{Tenenbaum} we require a lower bound for the quantity $m_g(X;T)$. For this we need to obtain a lower bound for 
\[
\sum_{\substack{z<p\leq w\\A_f(p)<0}}\frac1p
\]
uniformly in $2\leq z<w$.

As the underlying form satisfies the Ramanujan--Petersson Conjecture we have that $|A_f(p)|\leq n$ and so $|A_f(p)|\geq |A_f(p)|^2/n$. Consequently, using (\ref{uniformRS}), 
\begin{align}\label{est-1}
\sum_{z<p\leq w}\frac{|A_f(p)|}p&\geq\frac1n\sum_{z<p\leq w}\frac{|A_f(p)|^2}p \nonumber\\
&=\frac1n\sum_{z<p\leq w}\frac1p+O\left(\frac1{\log z}\right).
\end{align}
Thus, we have
\begin{align*}
\sum_{\substack{z<p\leq w\\A_f(p)>0}}\frac{A_f(p)}p+\sum_{\substack{z<p\leq w\\A_f(p)<0}}\frac{|A_f(p)|}p&=\sum_{z<p\leq w}\frac{|A_f(p)|}p\\
&\geq\frac1n\sum_{z<p\leq w}\frac1p-O_f\left(\frac1{\log z}\right).
\end{align*}
Using (\ref{uniformbound}) we have 
\[
\sum_{\substack{z<p\leq w\\A_f(p)>0}}\frac{A_f(p)}p-\sum_{\substack{z<p\leq w\\A_f(p)<0}}\frac{|A_f(p)|}p=O_f\left(\frac1{\log z}\right).
\]
Combining these estimates we get
\[
\min\left\{\sum_{\substack{z<p\leq w\\A_f(p)>0}}\frac{A_f(p)}p,\sum_{\substack{z<p\leq w\\A_f(p)<0}}\frac{|A_f(p)|}p\right\}\geq\frac1{2n}\sum_{z<p\leq w}\frac1p-O_f\left(\frac1{\log z}\right).
\]
As $|A_f(p)|\leq n$ under the Generalised Ramanujan--Petersson Conjecture, we deduce the estimates
\begin{align}\label{uniform_lb2}
\sum_{\substack{z<p\leq w\\A_f(p)>0}}\frac1p\geq\frac1{2n^2}\sum_{z<p\leq w}\frac1p-O_f\left(\frac1{\log z}\right)
\end{align}
and
\begin{align}\label{uniform_lb}
\sum_{\substack{z<p\leq w\\A_f(p)<0}}\frac1p\geq\frac1{2n^2}\sum_{z<p\leq w}\frac1p-O_f\left(\frac1{\log z}\right)
\end{align}
uniformly for $2\leq z<w$.

For $T\geq 1$ recall that we have defined
\[
m_g(X;T)=\min_{|t|\leq T}\sum_{p\leq X}\frac{|g(p)|-\Re(g(p)p^{-it})}p.
\]
We claim that $m_g(X;T)\gg \log\log X$ uniformly in $1\leq T\leq X$. This can be deduced in an elementary manner using Lemma \ref{pretentious}. Indeed, first observe, as $g$ is real-valued, that we have  
\begin{align*}
|g(p)|-\Re(g(p)p^{-it})=\begin{cases}
1-\cos(t\log p) & \text{if }A_f(p)>0 \\
1+\cos(t\log p) & \text{if }A_f(p)<0 \\
0 & \text{if }A_f(p)=0.
\end{cases}
\end{align*}
In the range $(\log X)^{-1/2}\leq |t|\leq X$ we trivially have $3\log(1+|t|\log X)\geq(3/2+o(1))\log\log X$, and the claimed lower bound follows from Lemma \ref{pretentious} (which is applicable by (\ref{uniform_lb2}) and (\ref{uniform_lb})) at once. To treat the range $|t|<(\log X)^{-1/2}$ we define the parameter $Y:=\exp(\sqrt{\log X}/4)<X$ and observe that for every $p\leq Y$ we have $|t|\log p\leq 1/4$. Thus, for any such prime we have $1+\cos(t\log p)\geq 1+\cos(-1/4)>0$. Now estimating trivially using non-negativity we have
\begin{align*}
\sum_{p\leq X}\frac{|g(p)|-\Re(g(p)p^{-it})}p &\geq\sum_{\substack{p\leq Y\\ A_f(p)<0}}\frac{|g(p)|-\Re(g(p)p^{-it})}p\\
&\geq\left(1+\cos\left(-\frac14\right)\right)\sum_{\substack{p\leq Y\\ A_f(p)<0}}\frac1p.
\end{align*}
We again obtain the required lower bound using (\ref{uniform_lb}) and observing that $\log\log Y\asymp\log\log X$. This completes the proof of the estimate $m_g(X;T)\gg\log\log X$. 

Next we apply Tenenbaum's result (Proposition \ref{Tenenbaum}) with the choice $T=(\log X)^4$. Condition (3) in Proposition \ref{Tenenbaum} is fulfilled by (\ref{MR_cond}). The other two conditions follow from $g(m)\in\{-1,0,1\}$. Using the previous observations, and the trivial upper bound $m_g(X;T)\ll\log\log X$, we have the estimate
\[
\frac{1+m_g(X;T)}{e^{m_g(X;T)}}\ll\frac{\log\log X}{e^{c_n\log\log X}} 
\]
for some constant $c_n$ depending only on $n$. Thus
\begin{align*}
\left|\sum_{m\sim X}g(m)\right|&\ll \left(\sum_{m\sim X}r(m)\right)\left(\frac{\log\log X}{(\log X)^{c_n}}+\frac1{\log X}\right) \\
&=o\left(\sum_{m\sim X}r(m)\right).
\end{align*}
Given (1), this completes the proof of (2), and thus also the proof of Theorem \ref{Main_theorem}. \qed

We finish this section by showing that the Generalised Ramanujan--Petersson Conjecture implies $\Delta_f(X)\gg_f (\log X)^{-1+1/n^2}$. Indeed, from (\ref{MR_cond}) we have
\[
\sum_{\substack{p\leq X\\A_f(p)=0}}\frac1p\leq\left(1-\frac1{n^2}\right)\log\log X+O_f(1). 
\]
Therefore
\begin{align*}
\log\Delta_f(X)&=\sum_{\substack{p\leq X\\A_f(p)=0}}\log\left(1-\frac1p\right) \\
&=-\sum_{\substack{p\leq X\\A_f(p)=0}}\frac1p+O(1)\\
&\geq -\left(1-\frac1{n^2}\right)\log\log X+O_f(1),
\end{align*}
which gives the desired estimate
\[
\Delta_f(X)\gg_f (\log X)^{-1+1/n^2}. 
\]

\section{Proof of Theorem \ref{secondMain_theorem}}

\noindent As in the proof of Theorem \ref{Main_theorem}, it suffices to show that 
\[
\sum_{m\sim X}r(m)\asymp_f X\Delta_f(X) \qquad \text{and} \qquad \sum_{m\sim X}g(m)=o_f\left(X\Delta_f(X)\right),
\]
where $g(m)$ and $r(m)$ are as in that proof. Again the latter bound follows from Proposition \ref{Tenenbaum} provided that the function $g$ is sufficiently non-pretentious, which in turn is proved using Lemma \ref{pretentious}. 

In the proof of Theorem \ref{Main_theorem} the Generalised Ramanujan--Petersson Conjecture was used to verify the hypotheses of Propositions \ref{MR} and \ref{Tenenbaum}, and Lemma \ref{pretentious}, together with a lower bound for $m_g(X;T)$. We begin with a general result showing that these conditions can be established with slightly weaker assumptions and afterwards verify them for self-dual $\mathrm{GL_2}$ and $\mathrm{GL}_3$ Hecke--Maass cusp forms unconditionally. 

Let us assume that the following estimates hold uniformly in $2\leq z<w$:
\begin{enumerate}
\item[] 
\begin{align}\label{cond1}
\sum_{z<p\leq w}\frac{|A_f(p)|^2}p=\sum_{z<p\leq w}\frac 1p+O_f\left(\frac1{\log z}\right);
\end{align}
\item[]
\begin{align}\label{cond2}
\sum_{z<p\leq w}\frac{A_f(p)}p=O_f\left(\frac1{\log z}\right);
\end{align}
\item[]
\begin{align}\label{cond3}
\sum_{z<p\leq w}\frac{|A_f(p)|^4}p\leq C_4\sum_{z<p\leq w}\frac 1p+O_f\left(\frac1{\log z}\right)
\end{align}
\end{enumerate}
for some absolute constant $C_4>0$. 

By the Cauchy--Schwarz inequality we have
\[
\left(\sum_{z<p\leq w}\frac{|A_f(p)|^2}p\right)^2\leq\left(\sum_{\substack{z<p\leq w\\ A_f(p)\neq 0}}\frac 1p\right)\left(\sum_{z<p\leq w}\frac{|A_f(p)|^4}p\right).
\]
Combining this with the assumptions (\ref{cond1}) and (\ref{cond3}) gives\footnote{Or alternatively use the elementary inequality (\ref{elem-ineq}) together with conditions (\ref{cond1}) and (\ref{cond3})}
\begin{align}\label{MR_cond2}
\sum_{\substack{z<p\leq w\\ A_f(p)\neq 0}}\frac 1p\geq\frac1{C_4}\sum_{z<p\leq w}\frac1p-O_f\left(\frac1{\log z}\right).
\end{align}
As noted in \cite[(13)]{Matomaki-Radziwill2020}, such an estimate implies  
\[
\sum_{m\sim X}r(m)\asymp_{f,C_4}X\Delta_f(X)
\]
with $r(m)=1_{A_f(m)\neq 0}$. Note also that arguing as at the end of Section $5$, (\ref{MR_cond2}) leads to the bound $\Delta_f(X)\gg_f (\log X)^{-1+1/C_4}$. Now it is enough to show that the conditions (\ref{cond1})-(\ref{cond3}) imply
\begin{align}\label{g-est}
\sum_{m\sim X}g(m)=o_f\left(X\Delta_f(X)\right).
\end{align}
We need to verify that both inequalities $A_f(p)>0$ and $A_f(p)<0$ hold for sufficiently many primes $p$. By interpolation we have   
\[
\left(\sum_{z<p\leq w}\frac{|A_f(p)|^2}p\right)^3\leq\left(\sum_{z<p\leq w}\frac{|A_f(p)|}p\right)^2\left(\sum_{z<p\leq w}\frac{|A_f(p)|^4}p\right).
\]
Combining this with the assumptions (\ref{cond1}) and (\ref{cond3}) gives
\[
\sum_{z<p\leq w}\frac{|A_f(p)|}p\geq\frac1{\sqrt{C_4}}\sum_{z<p\leq w}\frac 1p-O_f\left(\frac1{\log z}\right). 
\]
Using the assumption (\ref{cond2}) it follows that
\begin{align}\label{bound_for_neg_fc}
\sum_{\substack{z<p\leq w\\A_f(p)<0}}\frac{|A_f(p)|}p\geq\frac1{2\sqrt{C_4}}\sum_{z<p\leq w}\frac 1p-O_f\left(\frac1{\log z}\right)
\end{align}
and a similar inequality holds for the sum over primes with $A_f(p)>0$.

By the Cauchy--Schwarz inequality we have
\[
\left(\sum_{\substack{z<p\leq w\\A_f(p)<0}}\frac{|A_f(p)|}p\right)^2\leq\left(\sum_{\substack{z<p\leq w\\A_f(p)<0}}\frac1p\right)\left(\sum_{z<p\leq w}\frac{|A_f(p)|^2}p\right)
\]
from which we deduce, using (\ref{cond1}) and (\ref{bound_for_neg_fc}),
\[
\sum_{\substack{z<p\leq w\\A_f(p)<0}}\frac1p\geq\frac1{4C_4}\sum_{z<p\leq w}\frac 1p-O_f\left(\frac1{\log z}\right).
\]
Now the argument can be completed as in the proof of Theorem \ref{Main_theorem}. Indeed, the estimate above yields $m_g(X;T)\gg\log\log X$ uniformly in $1\leq T\leq X$. By Proposition \ref{Tenenbaum} this leads to (\ref{g-est}).

So to complete the proof it suffices to verify conditions (\ref{cond1})-(\ref{cond3}) unconditionally for self-dual $\mathrm{GL}_2$ and $\mathrm{GL}_3$ Hecke--Maass cusp forms. Let us start by establishing part (1) of the theorem. As noted in Section $4$, under Hypothesis H (which is known for self-dual $\mathrm{GL}_2$ Hecke--Maass cusp forms and more generally) conditions (\ref{cond1}) and (\ref{cond2}) hold. Let $\text{Sym}^2f$ be the symmetric square lift of $f$, which is a $\mathrm{GL}_3$ Hecke--Maass cusp form due to a result of Gelbart and Jacquet \cite{Gelbart-Jacquet1978} (the cuspidality of $\text{Sym}^2f$ here relies on the fact that any level $1$ Hecke--Maass cusp form on $\mathrm{GL}_2$ is non-dihedral). The Fourier coefficients of $f$ and $\text{Sym}^2f$ at primes are connected by the relation $A_f(p)^2=A_{\text{Sym}^2f}(p)+1$ and so in particular
\[
A_f(p)^4=A_{\text{Sym}^2f}(p)^2+2A_{\text{Sym}^2f}(p)+1. 
\]
Using Lemma \ref{automorphicPNT} with $\pi=\pi'=\text{Sym}^2f$ we have 
\[
\sum_{p\leq x}A_{\text{Sym}^2f}(p)^2\log p=x+O\left(xe^{-c\sqrt{\log x}}\right).
\]
Applying the same lemma with $\pi=\text{Sym}^2f$ and $\pi'=1$ (the trivial representation of $\mathrm{GL}_1$) gives
\[
\sum_{p\leq x}A_{\text{Sym}^2f}(p)\log p\ll xe^{-c\sqrt{\log x}}. 
\]
Combining these observations and the usual prime number theorem gives
\begin{align*}
\sum_{p\leq x}A_f(p)^4\log p &=\sum_{p\leq x}A_{\text{Sym}^2f}(p)^2\log p+2\sum_{p\leq x}A_{\text{Sym}^2f}(p)\log p+\sum_{p\leq x}\log p \\
&=2x+O\left(xe^{-c\sqrt{\log x}}\right). 
\end{align*}
This gives (\ref{cond3}) with $C_4=2$ by partial summation. We conclude using the bound $\Delta_f(X)\gg(\log X)^{-1+1/C_4}$.

The proof of part (2) requires more elaborate arguments. Let $\Pi$ be a fixed self-dual unitary cuspidal automorphic representation of $\mathrm{GL}_3(\A_\Q)$ associated with a $\mathrm{GL}_3$ Hecke--Maass cusp form $f$ with normalised Fourier coefficients $A_\Pi(m):=A(m,1)$. In particular, $\Pi$ is unramified everywhere. As before, put
\[
\Delta_\Pi(X):=\prod_{\substack{p\leq X\\A_\Pi(p)=0}}\left(1-\frac1p\right). 
\]
We will use a result of Ramakrishnan \cite[Theorem A]{Ramakrishnan2014} saying that every self-dual cuspidal automorphic representation of $\mathrm{GL}_3(\A_\Q)$ is a quadratic twist of an adjoint lift from $\mathrm{GL}_2$. So, there exists an unramified non-dihedral cuspidal representation $\sigma$ of $\mathrm{GL}_2(\A_\Q)$ and an unramified quadratic character $\nu$ (i.e. $\nu^2=1$) such that
\[
\Pi\simeq \text{Ad}(\sigma)\otimes\nu,
\]
where $\text{Ad}(\sigma):=\text{Sym}^2\sigma\otimes \omega_\sigma^{-1}$ with $\omega_\sigma$ being the central character of $\sigma$. 

We distinguish two cases depending on whether $\sigma$ is of tetrahedral/octahedral type or not (and throughout this section we assume that $\sigma$ is non-dihedral). Assume first that $\sigma$ is not of tetrahedral or octahedral type. For such representations we show the following. 

\begin{prop}\label{Pi-est}
Assume that $\Pi\simeq \mathrm{Ad}(\sigma)\otimes\nu$, where $\sigma$ is a non-dihedral unramified irreducible cuspidal automorphic representation of $\mathrm{GL}_2(\A_\Q)$ that is not of tetrahedral or octahedral type, and $\nu$ is a quadratic character. Then we have uniformly in $2\leq z<w$ that
\begin{enumerate}
\item \[
\sum_{z<p\leq w}\frac{|A_\Pi(p)|^2}p=\sum_{z<p\leq w}\frac 1p+O_\Pi\left(\frac1{\log z}\right); 
\]
\item \[
\sum_{z<p\leq w}\frac{A_\Pi(p)}p=O_\Pi\left(\frac1{\log z}\right);
\]
\item \[
\sum_{z<p\leq w}\frac{|A_\Pi(p)|^4}p=3\sum_{z<p\leq w}\frac 1p+O_\Pi\left(\frac1{\log z}\right).
\]
\end{enumerate}
\end{prop}

\begin{proof}
\noindent The first two bounds are direct consequences of Lemma \ref{automorphicPNT} and partial summation. For (3) we argue as follows. For simplicity, put $\rho:=\text{Ad}(\sigma)$ and $\tau:=\text{Sym}^4\sigma\otimes\omega_\sigma^{-2}$. By Kim's symmetric fourth power functoriality \cite{Kim2003}, $\tau$ is an automorphic representation for the group $\mathrm{GL}_5(\A_\Q)$. Furthermore, by the result of Kim and Shahidi \cite[Theorem 3.3.7]{Kim-Shahidi2002}, $\text{Sym}^4\sigma$ is cuspidal by our assumption on $\sigma$. 

Write the Satake parameters of $\sigma$ at a prime $p$ as $\alpha_p$ and $\beta_p$. Then we have the relations
\[
A_\rho(p)=\alpha_p\beta_p^{-1}+\alpha_p^{-1}\beta_p+1
\]
and 
\[
A_\tau(p)=\alpha_p^2\beta_p^{-2}+\alpha_p\beta_p^{-1}+1+\alpha_p^{-1}\beta_p+\alpha_p^{-2}\beta_p^2. 
\]
Thus a direct calculation gives
\begin{align*}
A_\rho(p)^2&=\alpha_p^2\beta_p^{-2}+\alpha_p^{-2}\beta_p^2+2\alpha_p\beta_p^{-1}+2\alpha_p^{-1}\beta_p+3\\
&=A_\tau(p)+A_\rho(p)+1.
\end{align*}
Squaring yields
\[
A_\rho(p)^4=A_\tau(p)^2+A_\rho(p)^2+1+2A_\tau(p)A_\rho(p)+2A_\tau(p)+2A_\rho(p).
\]
Noting that $A_\Pi(p)=\nu(p)A_\rho(p)$ and $\nu(p)^2=1$ we have  
\begin{align}\label{fourtpowerexp}
&\sum_{p\leq x}A_\Pi(p)^4\log p\\
&=\sum_{p\leq x}A_\tau(p)^2\log p+\sum_{p\leq x}A_\rho(p)^2\log p+\sum_{p\leq x}\log p \nonumber \\&\qquad\qquad\qquad +2\sum_{p\leq x}A_\tau(p)A_\rho(p)\log p+2\sum_{p\leq x}A_\tau(p)\log p+2\sum_{p\leq x}A_\rho(p)\log p. \nonumber
\end{align}
Using Lemma \ref{automorphicPNT} we have the estimates
\[
\sum_{p\leq x}A_\rho(p)^2\log p=x+O_\Pi\left(xe^{-c\sqrt{\log x}}\right) \quad \text{and} \quad \sum_{p\leq x}A_\rho(p)\log p\ll xe^{-c\sqrt{\log x}}. 
\]
For $\tau$, which is a $\mathrm{GL}_5$ representation, one needs Hypothesis H in order to apply part (2) of Lemma \ref{automorphicPNT}. This is one of the special higher degree cases in which this hypothesis is known\footnote{Of course this also follows from the recent work of Jiang \cite{Jiang2026}.} \cite{Kim2006}. As $\tau$ is cuspidal and self-dual, we have 
\[
\sum_{p\leq x}A_\tau(p)^2\log p=x+O_\sigma\left(\frac x{\log x}\right)
\]
by using part (2) of Lemma \ref{automorphicPNT}. Likewise, applying Lemma \ref{automorphicPNT} to the pair $(\pi,\pi')=(\tau,\rho)$ we have, as $\tau$ and $\rho$ have different degrees and hence cannot be twist equivalent, 
\[
\sum_{p\leq x}A_\tau(p)A_\rho(p)\log p=O_\sigma\left(\frac x{\log x}\right).
\]
Similarly we have 
\[
\sum_{p\leq x}A_\tau(p)\log p=O_\sigma\left(\frac x{\log x}\right)
\]
and finally by the ordinary prime number theorem
\[
\sum_{p\leq x}\log p=x+O\left(xe^{-c\sqrt{\log x}}\right).
\]
Using these in (\ref{fourtpowerexp}) gives
\[
\sum_{p\leq x}|A_\Pi(p)|^4\log p=3x+O_\Pi\left(\frac x{\log x}\right),
\] 
from which (3) follows by partial summation. This completes the proof of the proposition. 
\end{proof}

\noindent Thus conditions (\ref{cond1})-(\ref{cond3}) hold for $\Pi$ whenever the underlying $\mathrm{GL}_2$ representation $\sigma$ is neither tetrahedral nor octahedral. Now the proof can be completed as before in the case that $\sigma$ is any such representation. 

In the opposite case ($\sigma$ is tetrahedral or octahedral) again (\ref{cond1}) and (\ref{cond2}) hold by Lemma \ref{automorphicPNT}. However, for the fourth moment we need to proceed differently as $\text{Sym}^4\sigma$ is not cuspidal. Let us first assume that $\sigma$ is of octahedral type. In this case Kim and Shahidi \cite[Proposition 4.8]{Kim-Shahidi2002} have shown that $\sigma$ satisfies the Generalised Ramanujan--Petersson Conjecture. Thus using the estimate $|A_\Pi(p)|\leq 3$ we trivially have that $|A_\Pi(p)|^4\leq 9|A_\Pi(p)|^2$. Hence, using the second moment bound we deduce that  
\begin{align*}
\sum_{z<p\leq w}\frac{|A_\Pi(p)|^4}p&\leq 9\sum_{z<p\leq w}\frac{|A_\Pi(p)|^2}p\\
&\leq 9\sum_{z<p\leq w} \frac 1p+O_\Pi\left(\frac1{\log z}\right)
\end{align*}
uniformly in $2\leq z<w$. So (\ref{cond3}) holds also in this case. 

Assume finally that $\sigma$ is of tetrahedral type and that its Satake parameters at a prime $p$ are $\alpha_p$ and $\beta_p$. By \cite[Theorem 2.2.2]{Kim-Shahidi2002} we have an isomorphism
\begin{align}\label{isom}
\text{Sym}^3\sigma\otimes\omega_\sigma^{-1}\simeq (\sigma\otimes\mu)\boxplus(\sigma\otimes\mu^2)
\end{align}
for some non-trivial cubic character $\mu$ with $|\mu(p)|=1$ for every prime $p$. The Satake parameters at a prime $p$ on the left-hand side of (\ref{isom}) are 
\[
\frac{\alpha_p^2}{\beta_p},\,\alpha_p,\,\beta_p,\,\frac{\beta_p^2}{\alpha_p},
\]
and on the right-hand side they are $\mu(p)\alpha_p$, $\mu(p)\beta_p$, $\mu(p)^2\alpha_p$, $\mu(p)^2\beta_p$. As these multisets coincide, an elementary computation shows that $\max\{|\alpha_p|, |\beta_p|\}=1$. Since $|\alpha_p\beta_p|=1$, this gives $|\alpha_p|=|\beta_p|=1$. Thus the Generalised Ramanujan--Petersson Conjecture holds also in this case. In particular, $|A_\Pi(p)|\leq 3$ and the proof proceeds as in the octahedral case. The proof is completed. \qed 

\medskip

\noindent Finally, we deduce Corollaries \ref{Corollary} and \ref{Sign_changes}.

\subsection{Proof of Corollary \ref{Corollary}} Let $\mathcal S$ denote the set of points $x\in [X+H/4,2X]$ for which the estimate
\[
\mathcal N_f^\pm(x+H/2)-\mathcal N_f^\pm(x)\gg H\Delta_f(X) 
\]
holds. By Theorem \ref{Main_theorem} the set $\mathcal S$ has measure $X-H/4-o(X)\asymp X$. Choose a maximal $H$-separated set $x_1,...,x_J\in\mathcal S$. Maximality implies that the intervals $[x_j-H,x_j+H]$ cover the set $\mathcal S$ and consequently $J\gg X/H$. Set $I_j:=[x_j-H/4,x_j]$ for $j=1,...,J$. These intervals are pairwise disjoint and for $y\in I_j$ we have $(x_j,x_j+H/2]\subset (y,y+H]$. Therefore every positive or negative Hecke eigenvalues that is counted in the interval $(x_j,x_j+H/2]$ is also counted in the interval $(y,y+H]$, giving the desired result. \qed

\subsection{Proof of Corollary \ref{Sign_changes}} Consider the set 
\[
\mathcal G_f(X;H):=\{t\in[1,X-H]:\,\text{there exist integers }m_1,m_2\in(t,t+H]\,\text{ so that }A_f(m_1)<0<A_f(m_2)\}. 
\]
By inspecting the proof of Theorem \ref{Main_theorem} (which can be made unconditional for $n\in\{2,3\}$ by the proof of Theorem \ref{secondMain_theorem}) we see that the interval $(t,t+H]$ contains a sign change of $A_f(m)$ for $\gg_f X$ of the points $t\in[1,X-H]$ for sufficiently large $X$ when $H=C_f\Delta_f(X)^{-1}$, where $C_f$ is a sufficiently large constant depending on $f$. Using the trivial inequality
\[
|\mathcal G_f(X;H)|\leq H\cdot\#\{\text{sign changes of }A_f(m)\,\text{for }1\leq m\leq X\}
\]
it follows that the number of sign changes is $\gg X/H\gg X\Delta_f(X)$, as desired. On the other hand, the number of sign changes is bounded from above by $\#\{1\leq m\leq X:\,A_f(m)\neq 0\}\ll X\Delta_f(X)$, where the final estimate follows as in the proofs of Theorems \ref{Main_theorem} and \ref{secondMain_theorem}. This completes the proof. \qed

\section{Proofs of Theorems \ref{thirdMain_theorem} and \ref{fourthMain_theorem}}

\subsection{Proof of Theorem \ref{thirdMain_theorem}}
\noindent For part (1), using the automorphic prime number theorem with prime powers removed \cite[Theorem 8.4]{Jiang2026} and the pointwise estimate $|A_f(p)|\leq np^\vartheta$ we have 
\begin{align*}
x \ll \sum_{p\leq x}|A_f(p)|^2\log p \ll x^{2\vartheta}\cdot(\log x)\cdot\#\{p\leq x: A_f(p)\neq 0\}
\end{align*}
and so 
\begin{align}\label{HypH-conseq}
\#\{p\leq x: A_f(p)\neq 0\} \gg \frac{x^{1-2\vartheta}}{\log x}.
\end{align}
Using a method of Landau, Qu \cite[Theorem 1.1]{Qu2010} proved that the sequence $\{A_f(m)\}_{m=1}^\infty$ has infinitely many sign changes, and so in particular some coefficient $A_f(r)$ is negative. For every prime $p\leq x/r$ such that $(p,r)=1$ and $A_f(p)\neq 0$ multiplicativity gives that $A_f(p)$ and $A_f(pr)$ have opposite signs. Now the first part follows immediately from (\ref{HypH-conseq}).

For part (2) let us set 
\[
\mathcal P_f:=\{p\text{ prime}:A_f(p)\neq 0\}. 
 \]
As in the proof of Theorem \ref{Main_theorem}, the Generalised Ramanujan--Petersson Conjecture implies the uniform estimate
\begin{align*}
\sum_{\substack{z<p\leq w\\p\in\mathcal P_f}}\frac1p\geq\frac1{n^2}\sum_{z<p\leq w}\frac1p-O_f\left(\frac1{\log z}\right).
\end{align*}
Let $r$ be as above. Define the multiplicative function
\[
\ell(m):=\mu(m)^2 1_{p|m\Rightarrow p\in\mathcal P_f,\,(p,r)=1},
\]
i.e., $\ell(m)$ is the characteristic function of squarefree integers composed of primes $p$ coprime to $r$ at which $A_f(p)\neq 0$. 

We wish to apply Wirsing's theorem (Proposition \ref{Wirsing2}). To this end we define a multiplicative function $h$ by 
\begin{align*}
h(p):=\begin{cases}
\frac{|A_f(p)|^2}{n^2} & \text{if }(p,r)=1 \\
0 & \text{if }p|r
\end{cases}
\qquad \text{and} \qquad h(p^\nu):=0\quad\text{for }\nu\geq 2.
\end{align*}
As we assume the Generalised Ramanujan--Petersson Conjecture, $0\leq h(p)\leq 1$ and so $0\leq h(m)\leq \ell(m)$ for every $m$. Furthermore, using Lemma \ref{automorphicPNT} and partial summation we have
\begin{align*}
\sum_{p\leq x}\frac{h(p)\log p}p &=\frac1{n^2}\sum_{\substack{p\leq x\\ (p,r)=1}}\frac{|A_f(p)|^2\log p}p\\
&=\frac1{n^2}\sum_{p\leq x}\frac{|A_f(p)|^2\log p}p+O_r(1) \\
&=\left(\frac 1{n^2}+o(1)\right)\log x.
\end{align*}
Now Proposition \ref{Wirsing2} gives the lower bound
\begin{align}\label{Wirsing_lb}
\sum_{m\leq x}\ell(m)\geq \sum_{m\leq x}h(m)\gg \frac x{\log x}\prod_{p\leq x}\left(\sum_{j=0}^\infty\frac{h(p^j)}{p^j}\right)\gg x(\log x)^{1/n^2-1}. 
\end{align}
Here the rightmost estimate follows as
\[
\sum_{p\leq x}\frac{h(p)}p=\frac1{n^2}\log\log x+O_{f,r}(1),
\]
which is a direct consequence of (\ref{uniformRS}), and the fact that $h(p^j)=0$ for $j\geq 2$.

Consider the set
\[
\mathcal A:=\{m\in\N,\,\text{squarefree},\,(m,r)=1,\,p|m\Rightarrow p\in\mathcal P_f\}.
\]
Next we observe that for $m\in\mathcal A$ we have $A_f(m)\neq 0$ and $A_f(mr)=A_f(m)A_f(r)$. Note that any $m\in\mathcal A\cap[1,x/r]$ gives a pair $(A_f(m),A_f(mr))$ with opposite signs as $\text{sgn}(A_f(m)A_f(mr))=\text{sgn}(A_f(r))<0$. These pairs have disjoint indices as $(m,r)=1$ and $r>1$. As the cardinality of the set $\mathcal A\cap[1,x/r]$ is $\gg_r x(\log x)^{1/n^2-1}$ by (\ref{Wirsing_lb}), this concludes the proof. \qed 

\subsection{Proof of Theorem \ref{fourthMain_theorem} }

\noindent We first prove a general result and then specialise to $\mathrm{GL}_n$ with $n\leq 3$. Let $a:\N\longrightarrow\R$ be a multiplicative function and define a multiplicative set
\[
\mathcal M_a:=\left\{m\in\N:\,a(m)\neq 0\right\}.
\]
Suppose that there exists an absolute constant $\alpha>0$ so that 
\begin{align}\label{cond4}
\sum_{\substack{z<p\leq w\\p\in\mathcal M_a}}\frac1p\geq\alpha\sum_{z<p\leq w}\frac1p-O\left(\frac1{\log z}\right)
\end{align}
uniformly in $2\leq z<w$. Then we claim that 
\begin{align}\label{intersection_lb}
\#\left\{m\in\mathcal M_a\cap[X,2X]\right\}\gg X\Delta_a(X)
\end{align}
for
\[
\Delta_a(X):=\prod_{\substack{p\leq X \\ a(p)=0}}\left(1-\frac1p\right). 
\]
Indeed, this is a fairly direct consequence of \cite[Corollary 1.2 (i)]{Matomaki-Radziwill2020}, which says that for a given $2\leq h_0\leq X$ an interval $[x,x+h_0\Delta_a(X)^{-1}]$ of length $h_0\Delta_a(X)^{-1}$ contains $\gg_\alpha h_0$ elements $m\in\N$ such that $a(m)\neq 0$ for all but $\ll_{\alpha,\eps} Xh_0^{-1/2+\eps}$ of $x\sim X$. Choosing $h_0=C_\alpha$ to be a sufficiently large fixed constant depending on $\alpha$ it follows that for a positive proportion of $X\leq x\leq 2X-C_\alpha\Delta_a(X)^{-1}$ we have  
\[
\#\left\{m\in\mathcal M_a\cap[x,x+C_\alpha\Delta_a(X)^{-1}]\right\}\gg_\alpha 1.
\]
Integrating over $X\leq x\leq 2X-C_\alpha\Delta_a(X)^{-1}$ and noting that every integer is counted at most $O(\Delta_a(X)^{-1})$ times gives (\ref{intersection_lb}). 

Let us now suppose that $a(m)=A_f(m)\cdot 1_{(m,r)=1}$, where $f$ is a $\mathrm{GL}_2$ or self-dual $\mathrm{GL}_3$ Hecke--Maass cusp form and we have fixed $r\geq 2$ for which $A_f(r)<0$. Such an $r$ exists by a result of Qu \cite{Qu2010}. We consider the set 
\[
\mathcal M_{f,r}:=\left\{m\in\N:\, A_f(m)\neq 0,\,(m,r)=1\right\}.
\]
Note that condition (\ref{cond4}) holds with $\mathcal M_a$ replaced by $\mathcal M_{f,r}$ as in the proof of Theorem \ref{secondMain_theorem} because removing the finitely many primes dividing $r$ causes an error $\ll_r 1/z$. Thus we have
\[
\#\left\{m\sim X:\,m\in\mathcal M_{f,r}\right\}\gg_{f,r} X\Delta_f(X)
\]
by observing that 
\[
\Delta_a(X)=\Delta_f(X)\prod_{\substack{p\leq X \\ p|r \\ A_f(p)\neq 0}}\left(1-\frac1p\right)\asymp_r\Delta_f(X).
\]
For every $m\in\mathcal M_{f,r}$ we have $A_f(mr)=A_f(m)A_f(r)$ and so $A_f(m)$ and $A_f(mr)$ have opposite signs. Note that for $m\in\mathcal M_{f,r}$ with $m\leq X/r$ the pairs $(m,mr)$ are disjoint as $(m,r)=1$. Thus it follows that 
\[
\mathcal N_f^\pm(X)\gg_f X\Delta_f(X). 
\]
From the proof of Theorem \ref{secondMain_theorem} (see (\ref{MR_cond2}) and comments following it) we have that condition (\ref{cond4}) implies $\Delta_f(X)\gg(\log X)^{\alpha-1}$ and the condition holds with $\alpha=1/9$ for a self-dual $\mathrm{GL}_3$ Maass cusp form $f$ and $\alpha=1/2$ for $\mathrm{GL}_2$ Maass cusp forms, as we have seen. This completes the proof. \qed 

\section{Proofs of Theorems \ref{fifthMain_theorem} and \ref{sixthMain_theorem}}

\subsection{Proof of Theorem \ref{fifthMain_theorem}}
\noindent (1) Let us define the set
\[
\mathcal U:=\left\{x\sim X:\, \sum_{x< m\leq x+H}|A_f(m)|^2<\frac{r_fH_f(1)\cdot H}2\right\}.
\]
Note that for $x\in\mathcal U$ we have 
\[
\frac{r_fH_f(1)\cdot H}2<r_fH_f(1)\cdot H-\sum_{x< m\leq x+H}|A_f(m)|^2
\]
and so from part (1) of Lemma \ref{RS_average} and Chebyshev's inequality it immediately follows that 
\[
|\mathcal U|\ll\frac1{H^2}\int\limits_X^{2X}\left|\sum_{x< m\leq x+H}|A_f(m)|^2-r_fH_f(1)\cdot H\right|^2\,\mathrm d x\ll_\eps\frac{X^{1+2\vartheta+\eps}}{H},
\]
where we have fixed an arbitrarily small $\eps>0$. For $x\in\mathcal U^c$ we clearly have
\[
\sum_{x< m\leq x+H}|A_f(m)|\gg_\eps HX^{-\vartheta-\eps}
\]
using the trivial estimate $\max_{x\leq m\leq x+H}|A(m)|\ll_\eps X^{\vartheta+\eps}$.

Define another set
\[
\mathcal V:=\left\{x\sim X:\,\left|\sum_{x< m\leq x+H}A_f(m)\right|>r_fH_f(1)\cdot HX^{-\vartheta-2\eps}\right\}.
\]
Then by part (2) of Lemma \ref{RS_average} and a simple application of Chebyshev's inequality we have
\[
|\mathcal V|\ll_\eps\frac{X^{2\vartheta+4\eps}}{H^2}\int\limits_X^{2X}\left|\sum_{x< m\leq x+H}A_f(m)\right|^2\,\mathrm d x\ll_\eps\frac{X^{1+2\vartheta+5\eps}}H.
\]
Thus for $x\in \mathcal U^c\cap\mathcal V^c$ we have 
\[
\sum_{x< m\leq x+H}|A_f(m)|\pm\sum_{x< m\leq x+H}A_f(m)\gg_\eps HX^{-\vartheta-\eps},
\]
and consequently
\[
\sum_{\substack{x< m\leq x+H\\
\pm A_f(m)>0}}|A_f(m)|=\frac12\left(\sum_{x< m\leq x+H}|A_f(m)|\pm\sum_{x< m\leq x+H}A_f(m)\right)\gg_\eps HX^{-\vartheta-\eps},
\]
which gives the claim in view of the pointwise estimate $|A_f(m)|\ll_\eps m^{\vartheta+\eps}$. The exceptional set for which this does not hold has measure $\ll_\eps X^{1+2\vartheta+5\eps}/H$ by the bounds for the measures of $\mathcal U$ and $\mathcal V$. This completes the proof of the first part. 

(2) Let $0<\varepsilon<\eta/2$ be fixed. By Lemma \ref{GLH_conseq} we have the estimates
\[
\sum_{x< m\leq x+H}A_f(m)\ll_\eps x^{1/2+\eps} \qquad \text{and} \qquad \sum_{x< m\leq x+H}|A_f(m)|^2=r_fH_f(1)\cdot H+O_{f,\eps}\left(X^{1/2+\vartheta+\eps}\right).
\]
As $H\geq X^{1/2+\vartheta+\eta}$, it follows that 
\[
\sum_{x< m\leq x+H}|A_f(m)|^2\asymp H,
\]
and consequently 
\[
\sum_{x< m\leq x+H}|A_f(m)|\gg_\eps HX^{-\vartheta-\eps},
\]
using the pointwise bound $|A_f(m)|\ll_\eps X^{\vartheta+\eps}$. 

Furthermore, in the same range of $H$ we clearly have 
\[
\sum_{x< m\leq x+H}A_f(m)=o\left(\sum_{x< m\leq x+H}|A_f(m)|\right).
\]
Combining these estimates we deduce that 
\[
\sum_{x< m\leq x+H}|A_f(m)|\pm\sum_{x< m\leq x+H}A_f(m)\gg_\eps HX^{-\vartheta-\eps},
\]
and the proof can be concluded as in part (1).  \qed

\begin{rem}
One may wonder whether the results of this section can be improved using a bilinear structure as in \cite{Jaasaari2023}. Unfortunately, it seems that such a strategy does not yield an improvement without assuming the full Generalised Ramanujan--Petersson Conjecture. 
\end{rem}

\subsection{Proof of Theorem \ref{sixthMain_theorem}}
\noindent Fix $q>2$. By Shiu's bound (Lemma \ref{Shiu}) we have 
\begin{align}\label{Shiu_conseq2}
\sum_{x< m\leq x+H}|A_f(m)|^q\ll\frac H{\log x}\exp\left(\sum_{p\leq x}\frac{|A_f(p)|^q}p\right).
\end{align}
Under the Generalised Ramanujan--Petersson Conjecture we have $|A_f(p)|^{q}\leq n^{q-2}|A_f(p)|^2$. Together with the Rankin--Selberg asymptotics (\ref{uniformRS}) this gives
\[
\sum_{p\leq x}\frac{|A_f(p)|^q}p\leq n^{q-2}\sum_{p\leq x}\frac{|A_f(p)|^2}p=n^{q-2}\log\log x+O(1). 
\]
Plugging this into (\ref{Shiu_conseq2}) leads to 
\begin{align}\label{Shiu-conseq}
\sum_{x< m\leq x+H}|A_f(m)|^q\ll H(\log x)^{n^{q-2}-1}.
\end{align}
By H\"older's inequality we also have 
\begin{align}\label{Holder-appl} 
\sum_{x< m\leq x+H}|A_f(m)|^2\leq\left(\sum_{x< m\leq x+H}|A_f(m)|\right)^{\frac{q-2}{q-1}}\left(\sum_{x< m\leq x+H}|A_f(m)|^q\right)^{\frac1{q-1}}.
\end{align} 
As seen in the proof of part (2) of the previous theorem, thanks to the assumption\footnote{Recall that we assume the Generalised Ramanujan--Petersson Conjecture here.} $H\geq x^{1/2+\eta}$ we have that the left-hand side of (\ref{Holder-appl}) is $\asymp H$. Thus, recalling (\ref{Shiu-conseq}) we deduce the lower bound
\[
\sum_{x< m\leq x+H}|A_f(m)|\gg H(\log x)^{-(n^{q-2}-1)/(q-2)}.
\]
We also recall from the proof of Theorem \ref{fifthMain_theorem} (2) that 
\[
\sum_{x< m\leq x+H}A_f(m)=o\left(\sum_{x< m\leq x+H}|A_f(m)|\right)
\]
for $H\geq x^{1/2+\eta}$. 

Thus
\[
\sum_{x< m\leq x+H}|A_f(m)|\pm\sum_{x< m\leq x+H}A_f(m)\gg H(\log x)^{-(n^{q-2}-1)/(q-2)}.
\]
By the Cauchy--Schwarz inequality we have
\[
H^2(\log x)^{-2(n^{q-2}-1)/(q-2)}\ll\left(\mathcal N_f^\pm(x+H)-\mathcal N_f^\pm(x)\right)\sum_{x<m\leq x+H}|A_f(m)|^2.
\]
This implies the claim by recalling the estimate
\[
\sum_{x< m\leq x+H}|A_f(m)|^2\asymp H
\]
for $H\geq x^{1/2+\eta}$ and choosing $q$ to be arbitrarily close to $2$ as well as noting that 
\[
\lim_{q\rightarrow 2}\frac{n^{q-2}-1}{q-2}=\log n.
\]    
\noindent The proof is completed. \qed

\bibliography{Fourier_coefficients_same_sign}
\bibliographystyle{plain}
\end{document}